\documentclass[11pt, a4paper, reqno]{amsart}

\usepackage{amsfonts, amsthm, amsmath, amscd, amssymb}
\usepackage{hyperref}
\usepackage[T1]{fontenc}
\usepackage{ae}
\usepackage[all]{xy}

\renewcommand\AA{\mathbb{A}}
\newcommand\QQ{\mathbb{Q}}
\newcommand\CC{\mathbb{C}}

\newcommand\RR{\mathbb{R}}
\newcommand\TT{\mathcal{T}}
\newcommand\ZZ{\mathbb{Z}}
\newcommand\ZZp{\ZZ_{>0}}
\newcommand\ZZnz{\ZZ_{\ne 0}}
\newcommand\PP{\mathbb{P}}
\newcommand\xx{\mathbf{x}}
\newcommand\kk{\mathbf{k}}
\newcommand\EE{\mathcal{E}}

\newcommand\tx{\widetilde{X_4}}
\newcommand\dd{\,\mathrm{d}}
\newcommand\Ptwo{{\PP^2}}
\newcommand\Pfour{{\PP^4}}
\newcommand{\fbase}[5]{\ee^{({#1},{#2},{#3},{#4},{#5})}}
\newcommand{\congr}[3]{{#1} \equiv {#2}\ (\mathrm{mod}\ {#3})}
\DeclareMathOperator{\Pic}{Pic}
\DeclareMathOperator{\vol}{vol}
\DeclareMathOperator{\Spec}{Spec}
\DeclareMathOperator{\rank}{rank}
\newcommand\hcf{\gcd}
\newcommand\SI[2]{#1^{[#2]}}
\newcommand{\ex}[1]{#1}
\newcommand{\cp}[2]{{\hcf(#1,#2)=1}}
\newcommand{\ncp}[2]{{\hcf(#1,#2) > 1}}
\newcommand{\Aone}{{\mathbf A}_1}
\newcommand{\Atwo}{{\mathbf A}_2}
\newcommand{\Athree}{{\mathbf A}_3}
\newcommand{\Afour}{{\mathbf A}_4}
\newcommand{\Dfour}{{\mathbf D}_4}
\newcommand{\Dfive}{{\mathbf D}_5}
\newcommand{\tS}{{\widetilde S}}
\renewcommand{\le}{\leqslant}
\renewcommand{\ge}{\geqslant}
\newcommand{\ee}{\boldsymbol{\eta}}
\renewcommand{\aa}{\boldsymbol{\alpha}}
\newcommand\rto{\dashrightarrow}
\newcommand\e{\eta}
\newcommand\al{\alpha}
\newcommand\ep{\varepsilon}
\newcommand\RE{\Re e}
\newcommand\phis{\phi^*}
\newcommand\phid{\phi^\dagger}
\newcommand\Ga{\mathbb{G}_\mathrm{a}}

\newcommand\gonea{g_1^a}
\newcommand\goneb{g_1^b}
\newcommand\gtwoa{g_2^a}
\newcommand\gtwob{g_2^b}
\newcommand\gtwo{g_2}
\newcommand\Nzero{N_0}
\newcommand\Nonea{N_1^a}
\newcommand\Noneb{N_1^b}
\newcommand\Ntwoa{N_2^a}
\newcommand\Ntwob{N_2^b}
\newcommand\Rzero{R_0}
\newcommand\Ronea{R_1^a}
\newcommand\Roneb{R_1^b}
\newcommand\Rtwoa{R_2^a}
\newcommand\Rtwob{R_2^b}
\newcommand\thzero{\vartheta_0}
\newcommand\thonea{\vartheta_1^a}
\newcommand\thoneb{\vartheta_1^b}
\newcommand\thtwoa{\vartheta_2^a}
\newcommand\thtwob{\vartheta_2^b}
\newcommand\thtwo{\vartheta}

\renewcommand\rho{\varrho}
\renewcommand{\leq}{\leqslant}
\renewcommand{\geq}{\geqslant}

\newtheorem*{theorem}{Theorem}
\newtheorem{lemma}{Lemma}
\numberwithin{equation}{section}

\theoremstyle{definition}
\newtheorem*{ack}{Acknowledgements}

\begin{document}

\title[Manin's conjecture for a quartic del Pezzo surface] {Manin's
  conjecture for a quartic del Pezzo surface with $\Afour$ singularity}

\author{T.D. Browning}
\address{School of Mathematics, University of Bristol, Bristol BS8 1TW}

\email{t.d.browning@bristol.ac.uk}

\author{U. Derenthal} 

\address{Institut f\"ur Mathematik, Universit\"at
  Z\"urich, Winterthurerstrasse 190, 8057 Z\"urich, Switzerland}

\email{ulrich.derenthal@math.unizh.ch}


\begin{abstract}
  The Manin conjecture is established for a split singular del Pezzo surface
  of degree four, with singularity type $\Afour$.
\end{abstract}

\maketitle

\tableofcontents

\section{Introduction}\label{sec:intro}

The distribution of rational points on del Pezzo surfaces 
is a challenging topic that has enjoyed a surge of activity in recent
years. Guided by the largely unverified conjectures of Manin 
\cite{f-m-t} and his collaborators, the primary aim of this paper is
to investigate further the situation for 
split singular del Pezzo surfaces of degree $4$ in $\Pfour$, that are
defined over $\QQ$.
Our main achievement will be a proof of the Manin conjecture for 
the surface 
\begin{equation}
  \label{eq:surface}
   x_0x_1-x_2x_3 = x_0x_4+x_1x_2+x_3^2 = 0,
\end{equation}
which we denote by $S \subset \Pfour$. 
This surface contains a unique singularity of type $\Afour$ and exactly
three lines, all of which are defined over $\QQ$.

Let $U$ be the Zariski open subset formed by deleting the lines from $S$, and
let \[N_{U,H}(B) := \#\{x \in U(\QQ) \mid H(x) \le B\},\] for any $B\geq 1$.
Here $H$ is the usual height on $\Pfour$, in which the height $H(x)$ is
defined as $\max\{|x_0|,\dots,|x_4|\}$ for a point $x=(x_0:\ldots:x_4) \in
U(\QQ)$, provided that $\xx=(x_0,\ldots,x_4)$ has integral coordinates that
are relatively coprime. Bearing this in mind, the following is our principal
result.

\begin{theorem}
  We have \[N_{U,H}(B) =
  c_{S,H}B(\log B)^5 + O\big(B(\log B)^{5-2/7}\big),\] 
where \[c_{S,H} = \frac{1}{21600} \cdot
  \omega_\infty \cdot \prod_p\left(1-\frac 1 p\right)^6\left(1+\frac 6 p+\frac
    1{p^2}\right)\] and 
  \begin{equation}
    \label{eq:inf}
\omega_\infty =
  \int_{|t_2|,|t_2t_6t_7|,|t_7(t_6^3t_7+t_2^2)|,|t_6^2t_7|\le 1, 0 < t_6 \le
    1} \dd t_2 \dd t_6 \dd t_7.
  \end{equation}
\end{theorem}

It is easily checked that the surface $S$ is not toric, and we shall
see in Lemma~\ref{lem:A2} that it is not an equivariant compactification of
$\Ga^2$. Thus our result does not follow from the work of Tschinkel
and his collaborators \cite{MR1620682, clt}.

As the minimal desingularisation $\widetilde{X_d}$ of a split del Pezzo
surface $X_d\subset \PP^d$ of degree $d$ is the blow-up of $\Ptwo$ in $9-d$
points, it has Picard group $\Pic(\widetilde{X_d})\cong \ZZ^{10-d}$.
In the setting $d=4$, Manin's
conjecture \cite{f-m-t} therefore predicts that 
\begin{equation}
  \label{eq:manin-predict}
N_{U,H}(B) \sim \al(\tx) \omega_H(\tx) B(\log
B)^{5},
\end{equation}
as $B \rightarrow \infty$, where the exponent of $\log B$ is $\rank
\Pic(\tx)-1$. Moreover, the constants $\al(\tx)$ and $ \omega_H(\tx)$ are
those predicted by Peyre \cite{MR1340296}.  Note that the exponent of $\log B$
agrees with the statement of the theorem. We shall verify in
\S~\ref{sec:conjecture} that $c_{S,H}=\al(\tS) \omega_H(\tS)$ in this result.

An overview of progress relating to the Manin conjecture for arbitrary del
Pezzo surfaces can be found in the first author's survey \cite{gauss}.  The
present paper should be seen as a modest step on the path to its resolution
for the singular del Pezzo surfaces of degree $4$ that are split over $\QQ$.
According to the classification of such surfaces found in Coray and Tsfasman
\cite{c-t}, it transpires that there are $15$ possible singularity types for
split singular del Pezzo surfaces of degree $4$.  It follows from the work of
Batyrev and Tschinkel \cite{MR1620682}, la Bret\`eche and the first author
\cite{dp4-d5}, and the second author's joint work with Tschinkel
\cite{math.NT/0604193}, that the Manin conjecture is already known to hold for
$5$ explicit surfaces from this catalogue. In view of our theorem, which deals
with a surface of singularity type $\textbf{A}_4$, it remains to deal with the
split quartic del Pezzo surfaces that have singularity types
\begin{equation}
  \label{eq:anna-sit}
\text{$\textbf{A}_n$ for $n\in \{1,2,3\}$}, \quad 
3\textbf{A}_1, \quad 
\text{$\textbf{A}_1+\textbf{A}_n$ for $n\in \{1,2,3\}$}.
\end{equation}
Here one should note that there are two types of surfaces that have
singularity type $\textbf{A}_3$, one containing four lines and one containing
five lines. Similarly, there are two types that have $2\textbf{A}_1$
singularities.

The surface that we have chosen to focus on in the present investigation
satisfies the property that the cone of effective divisors associated to the
minimal desingularisation $\tS$ is not merely generated by the divisors that
form a basis for the Picard group $\Pic(\tS)$, but requires one further
divisor to generate it.  This leads to some additional considerations in the
proof, as we will see shortly.

The proof of the theorem uses a universal torsor.  For each split del Pezzo
surface of degree $d$, there is one (essentially unique) universal torsor,
which is always an open subset of a $(12-d)$-dimensional affine variety.  For
toric varieties, universal torsors are open subsets of affine space.
Salberger \cite{MR1679841} has shown how to establish Manin's conjecture using
universal torsors for split toric varieties defined over $\QQ$. As a step
towards handling non-toric del Pezzo surfaces that still have a relatively
simple universal torsor, the second author \cite{cox} has determined which del
Pezzo surfaces of degree at least 3 have a universal torsor that can be
described as a hypersurface in $\AA^{13-d}$.  Out of the singularity types in
\eqref{eq:anna-sit}, these include those surfaces of type $\Aone+\Atwo$,
$\Aone+\Athree$, $3\Aone$ and the $\Athree$ surface with five lines.  The
surfaces of type $\Dfive$ and $\Dfour$ considered in \cite{dp4-d5} and
\cite{math.NT/0604193} also belong to this class, as does the $\Afour$ surface
$S$ considered here.  In fact we will see in \S~\ref{sec:ut} that the
universal torsor for the present problem is an open subset of the hypersurface
  \begin{equation}
    \label{eq:ut}
\e_5\al_1 + \e_1\al_2^2 +
\e_3\e_4^2\e_6^3\e_7 = 0,
  \end{equation}
which is embedded in $\AA^9\cong \Spec \QQ[\e_1,\ldots,\e_7,\al_1,\al_2]$.
Note that one of the variables does not explicitly appear in the equation.

Our basic strategy is similar to the one used for the $\Dfive$ and $\Dfour$
quartic del Pezzo surfaces. The first step is to establish an explicit
bijection between the rational points outside the lines on $S$ and certain
integral points on the universal torsor. We adopt the approach of Tschinkel
and the second author \cite{math.NT/0604193} in order to obtain this bijection
in an elementary way, motivated by the structure of the minimal
desingularisation $\tS$ as a blow-up of $\Ptwo$ in five points.  The integral
points on the universal torsor are counted in \S~\ref{s:prelim}, using the
method developed by la Bret\`eche and the first author \cite{dp4-d5}.  The
torsor variables $\e_1, \dots, \e_7, \al_1, \al_2$ must satisfy \eqref{eq:ut},
together with certain coprimality and height conditions. The first step is to
fix the variables $\e_1, \dots, \e_7$ and to estimate the relevant number of
$\al_1, \al_2$ by viewing the equation as a congruence modulo $\e_5$. The
resulting estimate is then summed over the remaining variables.

The order in which we handle the remaining variables is crucial and subtle.
When it comes to summing over $\e_6$ and $\e_7$ we will run into trouble
controlling the overall contribution from the error term each time, because
both $\e_6$ and $\e_7$ can be rather big. Summing the number of $\al_1, \al_2$
over $\e_7$, for example, leads to an error term that we cannot estimate in a
way that is sufficiently small when summed over $\e_1, \dots, \e_5$ and large
values of $\e_6$.  In line with this we shall let the order of summation
depend on which of $\e_6$ or $\e_7$ has largest absolute value. When it comes
to summing the integral points on the universal torsor that satisfy $|\e_6|\ge
|\e_7|$, we sum first over $\e_6$ and then over $\e_7$. For the alternative
contribution we sum first over $\e_7$ and then over $\e_6$. This process leads
to two main terms that we put back together to get something of the general
shape
\begin{equation}
  \label{eq:MM}
M(\e_1, \dots, \e_5) := \omega_H(\tS)\cdot
\frac{B }{\e_1\e_2\e_3\e_4\e_5},
\end{equation}
where $\omega_H(\tS)$ is as in \eqref{eq:manin-predict}.
The final task is to sum this quantity over 
the remaining variables $\e_1, \dots, \e_5$.

While essentially routine, it is in this final analysis that a further
interesting feature of the proof of the theorem is
revealed.  For $\kk \in \ZZp^5$, define the simplex
\begin{equation}
  \label{eq:Pk}
P_\kk :=
\big\{(x_1, \dots, x_5) \in \RR^5 \mid x_i \ge 0,\quad k_1x_1+\dots+k_5x_5 \le
1\big\}, 
\end{equation}
whose volume is easily determined as
\[\vol(P_\kk) = \frac{1}{5!\cdot k_1\cdot k_2 \cdot k_3 \cdot k_4 \cdot
  k_5}.\] In \S~\ref{sec:conjecture} we will see that $\alpha(\tS) =
\vol(P_{(2,4,3,2,3)}) - \vol(P_{(3,6,4,2,5)})$, whence
\begin{equation}
  \label{eq:alpha}
  \alpha(\tS) 
= \frac{1}{5!\cdot 2\cdot 4\cdot 3\cdot 2\cdot 3}-\frac{1}{5!\cdot 3\cdot
    6\cdot 4\cdot 2\cdot 5}=\frac{1}{21600}.
\end{equation}
Returning to the summation of \eqref{eq:MM} over $\e_1,\ldots,\e_5 \in
\ZZp$, which is subject to $\e_1^2\e_2^4\e_3^3\e_4^2\e_5^3 \le B$,
it will transpire that there is a negligible contribution from those 
$\e_1,\ldots,\e_5$ for which 
$\e_1^3\e_2^6\e_3^4\e_4^2\e_5^5 > B$. 
Summing over the $\e_1,\ldots,\e_5 \in
\ZZp$ that are remaining therefore leads to the final main term
\[\big(\vol(P_{(2,4,3,2,3)}) - \vol(P_{(3,6,4,2,5)})\big)\cdot
\omega_H(\tS)B(\log B)^5,\]
as expected.  Thus the main term in 
the asymptotic formula is
really a difference of two main terms that conspire to give the
predicted value for $\al(\tS)$.
It would be interesting to see whether the same sort of phenomenon
occurs for other split del Pezzo surfaces of degree $4$, with
singularity type among the list \eqref{eq:anna-sit}.

\begin{ack} 
The authors are extremely grateful to the anonymous referee for his
careful reading of the manuscript and numerous helpful comments.
While working on this paper the first author was supported by EPSRC
grant number \texttt{EP/E053262/1}.
The second author was partially supported by a Feodor Lynen Research 
Fellowship of the Alexander von Humboldt Foundation. 
\end{ack}

\section{Calculation of Peyre's constant}\label{sec:conjecture}

In this section we wish to show that the value of the constant
$c_{S,H}$  obtained in our theorem is in
agreement with the prediction \eqref{eq:manin-predict} of Peyre
\cite{MR1340296}. Beginning with the value of $\omega_H(\tS)$, whose
precise definition we will not include here but which corresponds
to a product of local densities, we have
\begin{equation}\label{eq:om}
\omega_H(\tS)
= \omega_\infty \prod_p \Big(1 - \frac{1}{p}\Big)^6 \omega_p,
\end{equation}
where $\omega_\infty$ and $\omega_p$ are the real and $p$-adic
densities, respectively. The calculation of $\omega_p$ is routine and
leads to the conclusion that \[\omega_p= 1+\frac{6}{p}+\frac{1}{p^2}.\]
The reader is referred to \cite[\S~2]{dp4-d5} for an analogous calculation.
We now turn to the calculation of $\omega_\infty$, which needs to agree with
\eqref{eq:inf}.

Recall the equations \eqref{eq:surface} for the surface $S$, and write
$f_1(\xx)=x_0x_1-x_2x_3$ and $f_2(\xx) = x_0x_4+x_1x_2+x_3^2$.  To compute
$\omega_\infty$, we parametrise the points by writing $x_1,x_4$ as functions
of $x_0,x_2,x_3$. Thus we have
\[
x_1=\frac{x_2x_3}{x_0}, \quad
  x_4=-\frac{x_1x_2+x_3^2}{x_0}=-\frac{x_2^2x_3+x_0x_3^2}{x_0^2},
\]
and furthermore, 
\[\det
  \begin{pmatrix}
    \frac{\partial f_1}{\partial x_1} & \frac{\partial f_2}{\partial x_1}\\
    \frac{\partial f_1}{\partial x_4} & \frac{\partial f_2}{\partial x_4}
  \end{pmatrix}
  =\det
  \begin{pmatrix}
    x_0 & x_2 \\
    0 & x_0
  \end{pmatrix}
  =x_0^2.
\]
Since $\xx$ and $-\xx$ have the same image in $\Pfour$, we have
\begin{align*}
  \omega_\infty &= \frac 1 2
  \int_{|x_0|,\left|\frac{x_2x_3}{x_0}\right|,|x_2|,|x_3|,
    \left|\frac{x_2^2x_3+x_0x_3^2}{x_0^2}\right|\le 1} x_0^{-2} \dd x_0 \dd
  x_2 \dd x_3\\
  &= \frac 1 2
  \int_{|t_6|,|t_2t_6t_7|,|t_2|,|t_6^2t_7|,|t_2^2t_7+t_6^3t_7^2|\le 1} \dd t_2
  \dd t_6 \dd t_7,
\end{align*}
on carrying out the change of variables $x_0=t_6, x_2=t_2$ and $x_3=t_6^2t_7$.
But the range of integration is symmetric with respect to the transformation
$(t_2,t_6,t_7) \mapsto (t_2,-t_6,-t_7)$, and so we may restrict to the range
$t_6>0$. This therefore confirms the equality in \eqref{eq:inf}.

It remains to deal with the constant $\al(\tS)$ that appears in
\eqref{eq:manin-predict}.  As we've already commented, the Picard group
$\Pic(\tS)$ of $\tS$ has rank 6. Distinguished elements of $\Pic(\tS)$ are the
classes of irreducible curves with negative self intersection number. As
described in \S~\ref{sec:ut}, these are the classes of four exceptional
divisors $E_1, \dots, E_4$ coming from the $\Afour$-singularity of $S$ and the
transforms $E_5, E_6, E_7$ of the three lines on $S$.  By the work of the
second author \cite[\S~7]{cox}, $E_1, \dots, E_6$ form a basis of
$\Pic(\tS)$. In terms of this basis we have $E_7 = E_1+2E_2+E_3+2E_5-E_6$ and
$-K_\tS =2E_1+4E_2+3E_3+2E_4+3E_5+E_6$.

The convex cone in $\Pic(\tS)_\RR := \Pic(\tS) \otimes_\ZZ \RR$ generated by
classes of effective divisors is generated by $E_1, \dots, E_7$ (see
\cite[Theorem~3.10]{djt}). 
The intersection of its dual with the hyperplane
\[
\{x \in \Pic(\tS)_\RR \mid (x, -K_\tS) = 1\}
\] 
is a polytope $P$ whose volume is the constant $\alpha(\tS)$ defined
by Peyre \cite{MR1340296}. 
By definition 
\[
P = \left\{(x_1,\dots,x_6) \in \Pic(\tS)_\RR \mid
\begin{array}{l}
x_i \ge 0,\quad x_1+2x_2+x_3+2x_5-x_6\ge 0,\\
2x_1+4x_2+3x_3+2x_4+3x_5+x_6=1
\end{array}
\right\}.
\]
Eliminating the last coordinate shows that $P$ is isomorphic to 
\[
P' =
\left\{(x_1,\dots,x_5) \in \RR^5 \mid
\begin{array}{l}
  x_i \ge 0,\quad 2x_1+4x_2+3x_3+2x_4+3x_5 \le 1,\\
  3x_1+6x_2+4x_3+2x_4+5x_5 \ge 1
\end{array}\right\}.
\]
Analyzing the volume form with respect to which we must compute the volume of
$P$ in order to obtain $\alpha(\tS)$ (see \cite[Section~2]{djt}, for example),
we see that \[\alpha(\tS) = \vol(P')=\vol(P_{(2,4,3,2,3}) -
\vol(P_{(3,6,4,2,5)}),\] in the notation of \eqref{eq:Pk}.  This therefore
establishes \eqref{eq:alpha}.

An alternative approach to calculating $\alpha(\tS)$ is available to us
through recent work of Joyce, Teitler and the second author \cite{djt}. Recall
from \cite[Table~1]{alpha} that $\alpha(S_0) = 1/180$ for any non-singular
split del Pezzo surface $S_0$ of degree $4$. Since the order of the Weyl group
associated to the root system $\mathbf{A}_n$ is $(n+1)!$, as recorded in
\cite[Table~2]{djt}, so it follows from \cite[Theorem~1.3]{djt} that
\[\alpha(\tS)= \frac{1}{180} \cdot \frac{1}{5!} = \frac{1}{21600}.\]
This completes the verification that our theorem confirms the Manin conjecture
for the split $\Afour$ surface \eqref{eq:surface}.

\section{Arithmetic functions}\label{s:af}

In this section we present some elementary facts about certain
arithmetic functions and their average order, as required for our argument.
Define the multiplicative arithmetic functions
\begin{align*}
  \phis(n) := \prod_{p\mid n} \Big(1-\frac{1}{p}\Big), \quad 
  \phid(n) := \prod_{p\mid n} \Big(1+\frac{1}{p}\Big).
\end{align*}
Both of these functions have average order $O(1)$, and one has 
\begin{equation}
  \label{eq:phid}
  \sum_{n \le x} \frac{\phid(n)^j}{n} \ll_j \log x,
\end{equation}
for any $x>1$ and any $j\in\ZZp$.
To see this we note that $\phid(n)\leq \sum_{d\mid n} 1/d$, whence
\begin{align*}
  \sum_{n \le x} \frac{\phid(n)^j}{n} \leq \sum_{n\leq
  x}\frac{1}{n}\sum_{d_1,\ldots,d_j \mid n} \frac{1}{d_1\ldots d_j} 
\leq \sum_{d_1,\ldots,d_j=1}^\infty \frac{1}{d_1\ldots d_j
  [d_1,\ldots,d_j]} \sum_{e\leq x}\frac{1}{e},
\end{align*}
where $[d_1,\ldots,d_j]$ denotes the least common multiple of
$d_1,\ldots,d_j$. The required bound \eqref{eq:phid} then follows from
the estimate 
\[\sum_{d_1,\ldots,d_j=1}^\infty \frac{1}{d_1\ldots d_j [d_1,\ldots,d_j]} \leq
\sum_{d_1,\ldots,d_j=1}^\infty \frac{1}{(d_1\ldots d_j)^{1+1/j}} \ll_j 1.\]

For given positive integers $a,b$, our work will lead us to work with
the function 
\begin{equation}\label{eq:def_f}
  f_{a,b}(n) :=
  \begin{cases}
    \phis(n)/\phis(\gcd(n,a)), & \text{if $\cp n b$},\\
    0, & \text{if $\ncp n b$}.
  \end{cases}
\end{equation}
We begin by establishing the following result.

\begin{lemma}\label{lem:summation_f}
Let $I = [t_1,t_2]$, for $t_1<t_2$. 
Let $\alpha \in \ZZ$ such that $\cp \alpha q$.
Then we have
\[\sum_{\substack{n \in I \cap \ZZ\\\congr n \al q}} f_{a,b}(n) = \frac
  {t_2-t_1} q c_0 + O\big(2^{\omega(b)} \log|I|\big),\] 
where $|I| :=
2+\max\{|t_1|,|t_2|\}$ and 
\begin{equation}
  \label{eq:c0}
c_0 =
  \frac{\phis(b)}{\phis(\gcd(b,q))\zeta(2)} \prod_{p\mid abq}
  \Big(1-\frac{1}{p^2}\Big)^{-1}.
\end{equation}
\end{lemma}

\begin{proof}
We will follow the convention that $\mu(-n)=\mu(n)$ and $\mu(0)=0$.
We begin by calculating the Dirichlet convolution 
\[(f_{a,b}*\mu)(n)=\sum_{d\mid n}f_{a,b}(d)\mu(n/d)=
\prod_{\substack{p^\nu\|n\\\nu \ge 1}} \big(f_{a,b}(p^{\nu})-
f_{a,b}(p^{\nu-1})\big).\]
It is clear that $f_{a,b}(1)=1$ and 
\begin{equation*}
  f_{a,b}(p^{j})= f_{a,b}(p)=\left\{
    \begin{array}{ll}
      1-1/p, & \text{if $p\nmid ab$,}\\
      1, & \text{if $p\nmid b$ and $p\mid a$,}\\
      0, & \text{if $p\mid b$,}
    \end{array}
  \right.
\end{equation*}
for any $j\geq 1$. Hence it follows that 
\begin{equation*}
(f_{a,b}*\mu)(n)=
\left\{
\begin{array}{ll}
  \mu(n)\gcd(b,n)/|n|, & \text{if $\gcd(a,n)\mid b$,}\\ 
  0, & \text{otherwise.}
\end{array}
\right.
\end{equation*}
In particular
\begin{equation*}
  \sum_{n\leq N}|(f_{a,b}*\mu)(n)| \leq 
  \sum_{n\leq N} \frac{\gcd(b,n)|\mu(n)|}{|n|} 
  \ll  2^{\omega(b)}\log N,
\end{equation*}
for any $N>1$. 
Since $f_{a,b}=(f_{a,b}*\mu)*1$, we therefore deduce that 
\begin{align*}
\sum_{\substack{n \in I \cap \ZZ\\\congr n \al q}} f_{a,b}(n)
&=\sum_{\substack{d=1\\\gcd(d,q)=1}}^\infty (f_{a,b}*\mu)(d) 
\sum_{\substack{m\in d^{-1}I\cap \ZZ\\\congr {md} \al q}}1\\
&=\frac{t_2-t_1}{q}\sum_{\substack{d=1\\\gcd(d,q)=1}}^{\infty}
\frac{(f_{a,b}*\mu)(d)}{d}
+O\big(2^{\omega(b)}\log |I|\big).
\end{align*}
Here we have observed that the outer sum in the first line is really a sum
over $d\leq |I|$, making the previous bound applicable for dealing with the
error term. We have then extended the summation over $d$ to infinity, with
acceptable error.  Finally, it remains to observe that
\begin{align*}
\sum_{\substack{d=1\\\gcd(d,q)=1}}^{\infty}
\frac{(f_{a,b}*\mu)(d)}{d}=
\prod_{p}\Big(1-\frac{1}{p^2}\Big)
\prod_{p\mid abq}\Big(1-\frac{1}{p^2}\Big)^{-1}\prod_{\substack{p\mid
  b\\ p\nmid q}}\Big(1-\frac{1}{p}\Big)=
c_0,
\end{align*}
as required to complete the proof of the lemma.
\end{proof}

Rather than Lemma~\ref{lem:summation_f}, we will actually need a
corresponding estimate in which the summand is replaced by 
$f_{a,b}(n)g(n),$ for suitable real-valued functions $g$. 
This is supplied for us by the following result.

\begin{lemma}\label{lem:summation_f_g}
Let $I = [t_1,t_2]$, for $t_1<t_2$,
and let $g : I \to \RR$ be any function such
that $g$ has a continuous derivative on $I$ 
which changes its sign only $R_g(I)< \infty$ times on $I$.
Let $\alpha \in \ZZ$ such that $\cp \alpha q$.
Then we have
\[
\sum_{\substack{n \in I \cap \ZZ\\\congr n \al q}}
  f_{a,b}(n)g(n)
= \frac {c_0}q \int_I g(t) \dd t +
  O\big(2^{\omega(b)}\cdot (\log|I|)\cdot M_I(g)\big),\] 
with $c_0$ given by \eqref{eq:c0} and 
$M_I(g) := (1+R_g(I)) \cdot \sup_{t \in I}|g(t)|.$
\end{lemma}

\begin{proof}
  Let $S$ denote the sum that is to be estimated, and write 
\[M(t):=\sum_{\substack{n \le t\\
      \congr n \al q}} f_{a,b}(n),
\] 
for any $t>0$. By partial summation, 
\[S =
  M(t_2)g(t_2)-M(t_1)g(t_1) - \int_{t_1}^{t_2} M(t)g'(t) \dd t.\] 
An application of Lemma~\ref{lem:summation_f} reveals that 
$M(t) = c_0  t/ q + O(2^{\omega(b)} \log(2+|t|))$. Hence partial
integration yields 
\[S = \frac{c_0}{q} \int_I g(t) \dd t
  + O\big(2^{\omega(b)} \cdot (\log |I|)\cdot (|g(t_2)|+|g(t_1)|+\int_{t_1}^{t_2}
  |g'(t)|\dd t)\big).
\]
  Splitting $I$ into the $R_g$ intervals where $g'$ has constant sign
therefore completes the proof of the lemma.
\end{proof}

\section{The universal torsor}\label{sec:ut}

The purpose of this section is to establish a completely explicit bijection
between the rational points on the open subset $U$ of our $\mathbf{A}_4$
quartic del Pezzo surface $S$, and the integral points on the universal torsor
above $\tS$ which are subject to a number of coprimality conditions.  In doing
so we shall follow the strategy of the second author's joint work with
Tschinkel \cite{math.NT/0604193}.

Along the way we will introduce new variables $\eta_1,\ldots,\eta_7$
and $\al_1,\al_2$.  It will be convenient to henceforth write
\begin{equation}
  \label{eq:notat}
\ee=(\e_1,\dots,\e_5),\quad 
\ee'=(\e_1,\dots,\e_7),
\quad \aa =
(\al_1,\al_2).
\end{equation}
Furthermore, we will make frequent use of the notation 
\begin{equation}
  \label{eq:notat'}
\fbase{k_1}{k_2}{k_3}{k_4}{k_5} := \prod_{i=1}^5 \e_i^{k_i},  
\end{equation}
for any $(k_1,\dots,k_5) \in \QQ^5$.

In order to derive the bijection alluded to above, we must begin by collecting
together some useful information about the geometric structure of $S$, as
defined by equations~\eqref{eq:surface}.  By computing the Segre symbol of
$S$, the definition of which can be found in Hodge and Pedoe \cite{MR13:972c},
we see that $S$ contains exactly one singularity. This has type $\Afour$ and
is easily determined as $p=(0:0:0:0:1)$.  By the classification of singular
quartic del Pezzo surfaces found in Coray and Tsfasman
\cite[Proposition~6.1]{c-t}, $S$ contains exactly three lines.  Let us call
these lines $E''_5$, $E''_6$ and $E''_7$, where $E''_5$ and $E''_6$ intersect
in the singularity $p$, and $E''_7$ intersects $E''_6$ outside $p$. We easily
determine these lines as $E''_5=\{x_0=x_2=x_3=0\}$, $E''_6 =
\{x_0=x_1=x_3=0\}$ and $E''_7 = \{x_1=x_3=x_4=0\}$.

The projection $\xx \mapsto (x_1:x_3:x_4)$ from $E''_7$ is a birational map
$\phi: S \rto \Ptwo$, which maps 
\[
U:=S \setminus (E''_5 \cup E''_6 \cup E''_7) =
\{(x_0:\ldots:x_4)\in S \mid x_3 \ne 0\}
\] 
isomorphically to 
\[
\{(\al_2:\e_5:\al_1) \in
\Ptwo \mid \e_5 \ne 0, \al_1\e_5+\al_2^2 \ne 0\} \subset \Ptwo.
\] 
The inverse map is $\psi: \Ptwo \rto S$ given by
\begin{equation}
  \label{eq:map_psi}
  \psi: (\al_2:\e_5:\al_1) \mapsto (\e_5^3 : \al_2\e_7 : \al_2\e_5^2 : \e_5\e_7 :
\al_1\e_7),
\end{equation}
where $\e_7=-(\al_1\e_5+\al_2^2)$.

By \cite[Proposition~6.1, Diagram~12]{c-t}, blowing up the singularity
$p$ leads to a minimal desingularisation $\pi_0:\tS \to S$ containing four
$(-2)$-curves $E_1, \ldots, E_4$ (the four exceptional divisors obtained by
blowing up $p$) and three $(-1)$-curves $E_5, E_6, E_7$ (the strict transforms
of the lines $E_5'', E_6'', E_7''$ on $S$). The configuration of these $(-1)$-
and $(-2)$-curves on $\tS$ is described by
Figure~\ref{fig:dynkin},
where the number of edges between two curves is the intersection number, and
self intersection numbers are given as upper indices.
The divisors $A_1, A_2$ will be introduced momentarily.

\begin{figure}[ht]
  \centering
  \[\xymatrix{\SI{A_1}{1} \ar@{-}[rrrr] \ar@{-}[dd] \ar@{=}[dr] & & & & \SI{E_5}{-1} \ar@{-}[dr]\\
    & \SI{E_7}{-1} \ar@{-}[r] & \SI{E_6}{-1} \ar@{-}[r] & \ex{\SI{E_4}{-2}} \ar@{-}[r] & \ex{\SI{E_3}{-2}} \ar@{-}[r] & \ex{\SI{E_2}{-2}} \ar@{-}[dl]\\
    \SI{A_2}{0} \ar@{-}[rrrr] \ar@{-}[ur] & & & & \ex{\SI{E_1}{-2}}}\]
  \caption{Configuration of curves on $\tS$.}
  \label{fig:dynkin}
\end{figure}
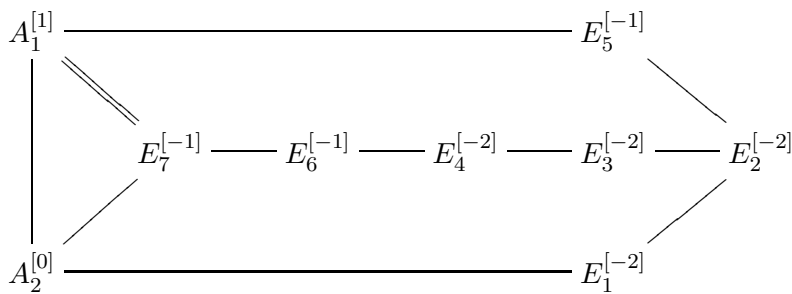

The surface $\tS$ is a blow-up $\pi: \tS \to \Ptwo$ in five points. While
there are several ways to construct $\tS$ as such a blow-up of $\Ptwo$, we
describe a map $\pi$ that is compatible with the map $\phi: S \rto \Ptwo$ in
the sense that $\phi \circ \pi_0: \tS \to S \rto \Ptwo$ coincides with $\pi$
where it is defined.  Such a map $\pi : \tS \to \Ptwo$ is obtained by
contracting $E_6, E_4, E_3, E_2, E_1$ on $\tS$ in this order. We choose the
same coordinates $(\al_2: \e_5: \al_1)$ on $\Ptwo$ as before. Then $\pi$ maps
$E_1, E_2, E_3, E_4, E_6$ to $(0:0:1)$. Furthermore, $E_7$ is the strict
transform of $E'_7=\{\e_7= -(\al_1\e_5+\al_2^2) = 0\} \subset \Ptwo$ and $E_5$
is the strict transform of $E'_5=\{\e_5=0\} \subset \Ptwo$ under $\pi$.

To describe which points on $\Ptwo$ we must blow up in order to recover $\tS$,
we introduce $A'_1=\{\al_1=0\} \subset \Ptwo$ and $A'_2=\{\al_2=0\} \subset
\Ptwo$. We note that its strict transforms $A_1, A_2$ under $\pi$ on $\tS$
intersect $E_1, \dots, E_7$ as described by Figure~\ref{fig:dynkin}, where
$A_1, A_2, E_7$ meet in one point which maps under $\pi_0$ to $(1:0:0:0:0) \in
S$. Given $E'_5, E'_7, A'_1, A'_2 \subset \Ptwo$ as above, we may now perform
the following sequence of five blow-ups to obtain $\tS$:
\begin{itemize}
\item blow up the intersection of $E_5, E_7, A_2$ to obtain $E_1$;
\item blow up the intersection of $E_1, E_5, E_7$ to obtain $E_2$;
\item blow up the intersection of $E_2, E_7$ to obtain $E_3$;
\item blow up the intersection of $E_3, E_7$ to obtain $E_4$;
\item blow up the intersection of $E_4, E_7$ to obtain $E_6$.
\end{itemize}
Here we have renamed $E'_i$ to $E_i$ and $A'_j$ to $A_j$, and we have
used the same names for a divisor and its strict
transform in each blow-up in the sequence. 
We proceed to establish the claim made in \S~\ref{sec:intro}.

\begin{lemma}\label{lem:A2}
  The surface $S$ is not an equivariant compactification of $\Ga^2$.
\end{lemma}

\begin{proof}
  To establish the lemma we assume for a contradiction that $S$ is of this
  type and apply the work of Hassett and Tschinkel \cite{h-t}.  If $S$ is an
  equivariant compactification of $\Ga^2$ then the map $\phi: S \rto \Ptwo$
  has to be $\Ga^2$-equivariant, resulting in an action of $\Ga^2$ on $\Ptwo$
  which leaves $E_7'=\{\e_7 = -(\al_1\e_5+\al_2^2)=0\}$ invariant.  However,
  we can check that the two distinct $\Ga^2$-structures on $\Ptwo$ (see
  \cite[Proposition~3.2]{h-t}) do not leave any irreducible quadric curve
  invariant.
\end{proof}

We are now ready to derive the promised bijection between $U(\QQ)$ and integral
points on the universal torsor lying above $\tS$. 
The map $\psi$ given by \eqref{eq:map_psi} 
induces a bijection 
\[
\psi_0 : (\al_1, \al_2, \e_5, \e_7)
\mapsto (\e_5^3, \al_2\e_7, \al_2\e_5^2, \e_5\e_7,
\al_1\e_7)
\] 
between 
\[\{(\aa, \e_5, \e_7) \in
\ZZ^2 \times \ZZp \times \ZZnz \mid
\al_1\e_5+\al_2^2+\e_7=0, \hcf(\al_1, \al_2, \e_5)=1\}\] and
\[
U(\QQ) = \{(x_0:\ldots:x_4) \in S(\QQ) \mid x_3 \ne 0\} \subset
S(\QQ).
\] 
Note that \[H(\psi_0(\al_1, \al_2, \e_5, \e_7)) = \frac{\max_{0\leq
    i\leq  4}|\psi_0(\al_1, \al_2, \e_5,
  \e_7)_i|}{\hcf(\{\psi_0(\al_1, \al_2, \e_5, \e_7)_i
  \mid 0\leq i\leq 4\})}.
\] 
Motivated by the sequence of
blow-ups above, we introduce new
variables 
\begin{equation*}
\begin{array}{lll}
\e_1:=\hcf(\al_2, \e_5, \e_7),
&\e_2:=\hcf(\e_1, \e_5, \e_7), &\e_3:=\hcf(\e_2,
\e_7), \\
\e_4:=\hcf(\e_3, \e_7),&\e_6:=\hcf(\e_4,\e_7), &
\end{array}
\end{equation*}
and in each step transform and rename the previous variables
accordingly.

Observe that this gives a bijection
\[(\ee', \aa) \mapsto
(\fbase 2 4 3 2
3 \e_6, \fbase 1 1 1 1 0 \e_6\e_7\al_2, \fbase 2 3 2 1 2 \al_2, \fbase 1
2 2 2 1 \e_6^2\e_7, \e_7\al_1),\]
which we call $\Psi$, between 
\[\TT := \left\{(\ee', \aa) \in
  \ZZp^6 \times \ZZnz \times \ZZ^2 \bigg{|}
\begin{aligned}
  &\e_5\al_1 + \e_1\al_2^2 +
\e_3\e_4^2\e_6^3\e_7 = 0\\ &\text{coprimality conditions hold}
\end{aligned}\right\}\]
and $U(\QQ)$. The coprimality conditions are described by the extended Dynkin
diagram of $E_1, \dots, E_7, A_1, A_2$ in Figure~\ref{fig:dynkin}, following
the rule that any of the variables $\eta_i,\al_j$ are coprime if and only if
there is no line connecting the divisors $E_i, A_j$ in the Dynkin
diagram. Once taken in conjunction with the equation
\begin{equation*}
T(\ee',\aa)=\e_5\al_1 + \e_1\al_2^2 +
\e_3\e_4^2\e_6^3\e_7 = 0,
\end{equation*}
that is satisfied by the elements of $\TT$, it is easily checked that
the coprimality conditions can be rewritten as
\begin{align}
  \label{eq:cpal1} &\cp{\al_1}{\e_2\e_6} \\
  \label{eq:cpal2} &\cp{\al_2}{\e_2\e_3\e_4} \\
  \label{eq:cpe6} &\cp{\e_6}{\e_1\e_2\e_3\e_5} \\
  \label{eq:cpe7} &\cp{\e_7}{\e_1\e_2\e_3\e_4\e_5} \\
  \label{eq:cpe} &\cp{\e_1}{\e_3\e_4\e_5},\ \cp{\e_2}{\e_4},\
  \cp{\e_5}{\e_3\e_4}.
\end{align}
In particular it follows that 
$H(\Psi(\ee',\aa)) = \max_{0\leq i\leq 4}|\Psi(\ee',\aa)_i|,$ 
since the five coordinates of
$\Psi(\ee',\aa)$ are necessarily coprime for $(\ee',\aa) \in \TT$.
The height conditions may therefore be written as
\begin{equation}
  \label{eq:anna-clap}
\max\Big\{ \begin{array}{l}
|\fbase 2 4 3 2 3 \e_6|, |\fbase 1 1 1 1 0 \e_6\e_7\al_2|, \\
|\fbase 2 3 2 1 2 \al_2|, |\fbase 1
2 2 2 1 \e_6^2\e_7|, |\e_7\al_1|
\end{array}
\Big\} \leq B.
\end{equation}

The equation $T(\ee',\aa)= 0$ is an embedding of the universal torsor over
$\tS$ in $\AA^{9}$.  Our argument so far has given us a parametrisation of
rational points of bounded height in the complement $U$ of the lines in
$S$. This will play a pivotal role in our proof of the theorem.

\section{The main argument}\label{s:prelim}

In this section we give an overview of the proof of 
the theorem, and make our final preparations for its proof. 
Recall the notation introduced in \eqref{eq:notat} and
\eqref{eq:notat'} for $\ee, \aa$ and
$\fbase{k_1}{k_2}{k_3}{k_4}{k_5}$. 
We define the quantities 
\begin{align*}
    Y_0&:=\left(\frac{\fbase 2 4 3 2 3}{B}\right)^{1/5},\\
    Y_2&:=\left(\frac B{\fbase{2}{-1}{-2}{-3}{-2}}\right)^{1/5},\\
    Y_6&:=Y_0^{-1},\\
    Y_7&:=\left(\frac B{\fbase{-3}{-6}{-2}{2}{-7}}\right)^{1/5},
\end{align*}
which clearly depend only on $\ee$ and $B$.  Using the equation
$T(\ee',\aa)=0$, a little thought reveals that we may write the height
condition \eqref{eq:anna-clap} as
\begin{align}
  \label{eq:H0} &|Y_0^4(\e_6/Y_6)| \le 1, \\
  \label{eq:H1} &|Y_0^2(\e_6/Y_6)(\e_7/Y_7)(\al_2/Y_2)| \le 1,\\
  \label{eq:H2} &|Y_0^4(\al_2/Y_2)| \le 1 ,\\
  \label{eq:H3} &|Y_0^2(\e_6/Y_6)^2(\e_7/Y_7)| \le 1,\\
  \label{eq:H4} &|(\e_7/Y_7)((\e_6/Y_6)^3(\e_7/Y_7)+Y_0^2(\al_2/Y_2)^2)|
  \le 1,
\end{align}
with $\e_1, \dots, \e_6 > 0$.  For example, eliminating $\al_1$ from
$|\e_7\al_1|\le B$ using $T(\ee',\aa)=0$ gives \eqref{eq:H4}.  It follows from
the contents of \S~\ref{sec:ut} that $N_{U,H}(B)$ is equal to the number of
$(\ee', \aa) \in \ZZp^6 \times \ZZnz \times \ZZ^2$ such that \eqref{eq:ut}
holds, with \eqref{eq:cpal1}--\eqref{eq:cpe} and \eqref{eq:H0}--\eqref{eq:H4}
all holding.  As indicated in the introduction it will be necessary to follow
different arguments according to which of $\eta_6$ or $|\eta_7|$ is biggest in
the summation over the variables $\ee'$. Accordingly, we write $N_a(B)$ for
the overall contribution to $N_{U,H}(B)$ from $(\ee', \aa)$ such that
\begin{equation}
  \label{eq:Ta}
\eta_6\geq  |\eta_7|,
\end{equation}
and $N_b(B)$ for the remaining contribution from $(\ee', \aa)$ such that 
\begin{equation}
  \label{eq:Tb}
\eta_6< |\eta_7|.
\end{equation}
These quantities will be estimated in \S~\ref{s:a} and \S~\ref{s:b},
respectively.

Let us now recall the broad outlines of our approach to estimating $N_{a}(B)$
and $N_{b}(B)$, as discussed in \S~\ref{sec:intro}. Thus the idea is to view
the torsor equation \eqref{eq:ut} as a congruence modulo $\e_5$, in order to
take care of the summation over the variable $\al_1$. In \S~\ref{s:a12} we
shall use this strategy to count the total number of permissible
$\aa=(\al_1,\al_2)$.  This will lead to a preliminary estimate for both
$N_{a}(B)$ and $N_{b}(B)$, since it will make no difference whether
\eqref{eq:Ta} or \eqref{eq:Tb} holds.  It will then remain to sum this
estimate over all of the remaining variables $\ee'$.  We will estimate the
overall contribution from the error term in \S~\ref{s:a12}. For the treatment
of the main term, however, we will need to treat the cases in which
\eqref{eq:Ta} or \eqref{eq:Tb} holds differently.  In estimating $N_{a}(B)$,
we will sum the main term over $\eta_6$ and then over $\eta_7$. This will be
undertaken \S~\ref{s:a}. Alternatively, to estimate $N_{b}(B)$, we will sum
the main term over $\eta_7$ and then over $\eta_6$.  This will be the object
of \S~\ref{s:b}. Finally, in \S~\ref{s:final} we will recombine our estimates
and sum over the remaining variables $\ee=(\eta_1,\ldots,\eta_5)$.

\subsection{Real-valued functions}

In estimating $N_a(B)$ and $N_{b}(B)$ we will meet a number of real-valued
functions, whose basic properties it will be crucial to understand.  Let
\begin{equation}\label{eq:h}
h(t_0,t_2,t_6,t_7):=\max
\{
|t_0^4t_6|,|t_0^2t_2t_6t_7|,|t_0^4t_2|,|t_0^2t_6^2t_7|,|t_7(t_6^3t_7+t_0^2t_2^2)|
\}
\end{equation}
Bearing this notation in mind, one notes that the height conditions in 
 \eqref{eq:H0}--\eqref{eq:H4} are equivalent to $h(Y_0, \al_2/Y_2,
\e_6/Y_6, \e_7/Y_7) \le 1$.  Finally, it is easy to see that 
  \[
\omega_\infty = \int_{h(1,t_2,t_6,t_7) \le 1, t_6>0} \dd t_2 \dd t_6
  \dd t_7,\]
where $\omega_\infty$ is given by \eqref{eq:inf}.

We define the real-valued functions
\begin{align}
  \label{eq:g0} g_0(t_0,t_6,t_7)&:= \int_{h(t_0,t_2,t_6,t_7) \le 1} 1 \dd t_2\\
  \label{eq:g1g} \gonea(t_0,t_7;\ee; B)&:= \int_{Y_6t_6 \ge
    |Y_7t_7|, t_6>0} g_0(t_0,t_6,t_7) \dd t_6\\
  \label{eq:g1l} \goneb(t_0,t_6;\ee; B)&:= \int_{|Y_7t_7|>\max\{Y_6t_6,1\}}
  g_0(t_0,t_6,t_7) \dd t_7\\
  \label{eq:g2g}
  \begin{split}
    \gtwoa(t_0;\ee; B)&:= \int_{h(t_0,t_2,t_6,t_7) \le 1, Y_6t_6 \ge
      |Y_7t_7|>1} \dd t_2 \dd t_6 \dd t_7\\&= \int_{|t_7|>1/Y_7}
    \gonea(t_0,t_7;\ee; B) \dd t_7
  \end{split}\\
  \label{eq:g2l}
  \begin{split}
    \gtwob(t_0;\ee; B)&:= \int_{h(t_0,t_2,t_6,t_7) \le 1, |Y_7t_7| >
      \max\{Y_6t_6,1\}, t_6>0} \dd t_2 \dd t_6 \dd t_7 \\&=
    \int_0^\infty \goneb(t_0,t_6;\ee; B) \dd t_6
  \end{split}
\end{align}
We clearly have 
\begin{equation}
  \label{eq:anna-fly}
\begin{split}
\gtwo(t_0;\ee;B)&:=\gtwoa(t_0;\ee;B)+\gtwob(t_0;\ee;B)  \\
&=
\int_{h(t_0,t_2,t_6,t_7) \le 1, |Y_7t_7|>1, t_6>0} \dd t_2 \dd t_6 \dd t_7.
\end{split}
\end{equation}
Finally, we define \[G_2(t_0):= \int_{h(t_0,t_2,t_6,t_7) \le 1, t_6>0} \dd t_2
\dd t_6 \dd t_7.\] The function $G_2:\RR_{>0}\rightarrow \RR$ is intimately
related to the real density $\omega_\infty$, as the following result shows.

\begin{lemma}\label{lem:omega_infty}
We have \[G_2(t_0) = \frac{\omega_\infty}{t_0^2}.\]
\end{lemma}

\begin{proof}
This result follows on making the  change of variables 
\[t_2 = T_2t_0^{-4},\quad t_6 =
  T_6t_0^{-4},\quad t_7 = T_7t_0^{6}.
\] 
Under this transformation one therefore obtains 
\[G_2(t_0) = \frac 1 {t_0^2}
  \int_{h(t_0,T_2t_0^{-4},T_6t_0^{-4},T_7t_0^6)\le 1, T_6>0} \dd T_2 \dd T_6 \dd
  T_7,
\] 
where $h(t_0,T_2t_0^{-4},T_6t_0^{-4},T_7t_0^6)=h(1,T_2,T_6,T_7)$ is
independent of $t_0$.
\end{proof}

During the course of our main argument it will be absolutely critical
to control the size of the functions 
\eqref{eq:g0}--\eqref{eq:g1l}, as $t_0,t_6,t_7$ vary.
We may and shall assume that $t_0,t_6,|t_7|$ take only positive values.

\begin{lemma}\label{lem:bound_g}
Let $\ee\in \ZZp^5$ be given. Then the following
hold:
\begin{enumerate}
\item\label{it:bound_gzero} $g_0(t_0,t_6,t_7) \ll \frac{1}{t_0|t_7|^{1/2}}$.
\item\label{it:bound_gonea} $\gonea(t_0,t_7; \ee;B) \le
    \int_0^\infty g_0(t_0,t_6,t_7) \dd t_6 \ll\min\big\{
    \frac{1}{t_0|t_7|^{7/6}}, \frac{1}{t_0^8}\big\}$.
\item\label{it:bound_goneb} $\goneb(t_0,t_6; \ee;B) \le
    \int_{-\infty}^\infty g_0(t_0,t_6,t_7) \dd t_7 \ll
    \frac{1}{t_0t_6^{3/4}}$.
  \end{enumerate}
\end{lemma}

\begin{proof}
Recall the definition \eqref{eq:h} of $h$.
The upper bound $O(t_0^{-8})$ that appears in~\eqref{it:bound_gonea}
is easy. Indeed, it follows from 
  the inequality $h(t_0,t_2,t_6,t_7) \le 1$ that $|t_2| \le 1/t_0^4$
  and $|t_6| \le 1/t_0^4$.

  For the remaining statements, we distinguish the case $|t_6^3t_7^2| \le 2$
  and its opposite. Note that the inequality $h(t_0,t_2,t_6,t_7) \le 1$
  implies
  \begin{equation}
    \label{eq:anna-cries}
|t_6^3t_7^2+t_0^2t_2^2t_7| \le 1.
  \end{equation}
  Let us begin with the first case, in which case $|t_0^2t_2^2t_7| \le
  3$. We therefore obtain
\[
t_2 \ll \frac{1}{t_0|t_7|^{1/2}}, \quad t_6 \ll \frac{1}{|t_7|^{2/3}},
  \quad t_7 \ll \frac{1}{t_6^{3/2}}.
\] 
The first of these inequalities implies statement~\eqref{it:bound_gzero}, the
first and second imply the first bound in statement~\eqref{it:bound_gonea},
and finally, integrating the bound for $g_0(t_0,t_6,t_7)$ from
statement~\eqref{it:bound_gzero} over $ t_7 \ll 1/t_6^{3/2}$ gives
statement~\eqref{it:bound_goneb}.

In the second case $|t_6^3t_7^2| > 2$, the inequality \eqref{eq:anna-cries}
implies $t_7 < 0$ and
\[\frac{t_6^3t_7^2-1}{t_0^2|t_7|} \le t_2^2 \le
\frac{t_6^3t_7^2+1}{t_0^2|t_7|}.\] Note that the condition $\sqrt{x} \le t_2
\le \sqrt{x+y}$ describes an interval for $t_2$ of length
$O(y/x^{1/2})$. Here, $x = (t_6^3t_7^2-1)/(t_0^2|t_7|) \ge
t_6^3|t_7|/(2t_0^2)$ and $y=2/(t_0^2|t_7|)$, whence
\[g_0(t_0,t_6,t_7) \ll \frac{1}{t_0t_6^{3/2}|t_7|^{3/2}}.\] The inequality
$t_6 > 2^{1/3}/|t_7|^{2/3}$ implies statement~\eqref{it:bound_gzero} and
integrating over $t_6 > 2^{1/3}/|t_7|^{2/3}$ results in the first bound in
statement~\eqref{it:bound_gonea}. Finally, integrating over $|t_7| >
2^{1/2}/t_6^{3/2}$ gives statement~\eqref{it:bound_goneb}.
\end{proof}

\subsection{Estimating $N_a(B)$ and $N_b(B)$ --- first step}\label{s:a12}

We are now ready to begin our estimation of $N_a(B)$ and $N_{b}(B)$ in
earnest.  In what follows, we always have $\e_1, \dots, \e_6 \in \ZZp$ and
$\e_7 \in \ZZnz$.
 
For fixed $\ee'=(\ee, \e_6, \e_7)$ subject to 
the coprimality conditions \eqref{eq:cpe6},
\eqref{eq:cpe7} and \eqref{eq:cpe}, 
we let $\Nzero:=\Nzero(\ee';B)$ be the total number of $\al_1,
\al_2\in\ZZ$ which satisfy the equation \eqref{eq:ut},  subject to 
$h(Y_0,\al_2/Y_2,\e_6/Y_6,\e_7/Y_7) \le 1$
and the coprimality conditions \eqref{eq:cpal1} and \eqref{eq:cpal2}.
Employing a M\"obius inversion for \eqref{eq:cpal1}, we obtain
\begin{equation*}
  \Nzero 
  =\sum_{k_1\mid \e_2\e_6} \mu(k_1)\#\left\{\al_2 \mid
\begin{array}{l}
\congr{\al_2^2\e_1}{-\e_3\e_4^2\e_6^3\e_7}{k_1\e_5},\\ 
h(Y_0, \al_2/Y_2,\e_6/Y_6, \e_7/Y_7) \le 1,\\ 
\text{\eqref{eq:cpal2} holds}
\end{array}
\right\}.
\end{equation*}
It is easy to see that the summand vanishes unless $\cp{k_1}{\e_1\e_3\e_4}$.
Indeed, if $p\mid k_1,\e_1$ then $p\mid \e_1,\e_3\e_4\e_6\e_7$, which is
forbidden, and furthermore, if $p\mid k_1,\e_3\e_4$ then $p\mid
\e_3\e_4,\al_2\e_1$, which is also forbidden.

Let $k_1$ be a squarefree divisor of $\e_2\e_6$. Since
$\cp{\e_2}{\e_6}$, we can write $k_1=k_{12}k_{16}$ with $k_{12}\mid\e_2$
and $k_{16}\mid\e_6$. Furthermore such a representation is
unique. Writing $\e_6=k_{16}\e_6'$ we therefore obtain
\[
\Nzero=
\sum_{\substack{k_{16}\mid \e_6, k_{12}\mid \e_2\\\cp{k_{12}k_{16}}{\e_1\e_3\e_4}}}
\mu(k_{12})\mu(k_{16})N_0(k_{12},k_{16})
\]
where
\[
N_0(k_{12},k_{16}):=
\#\left\{\al_2 \mid
\begin{array}{l}
\congr{\al_2^2\e_1}{-\e_3\e_4^2k_{16}^3\e_6'^3\e_7}{k_{12}k_{16}\e_5},\\ 
h(Y_0, \al_2/Y_2,\e_6/Y_6, \e_7/Y_7) \le 1,\\ 
\text{\eqref{eq:cpal2} holds}
\end{array}
\right\}.
\]
In view of the congruence we have $k_{16}\mid \al_2^2\e_1$, whence $k_{16}
\mid \al_2$ since $\cp{k_{16}}{\e_1}$ and $k_{16}$ is squarefree. Writing
$\al_2=k_{16}\al_2'$, we divide through the congruence by $k_{16}$ to obtain
\[
\congr{\al_2'^2k_{16}\e_1}{-\e_3\e_4^2k_{16}^2\e_6'^3\e_7}{k_{12}\e_5}.
\]
Using the relation $\cp{\e_6}{\e_2\e_5}$, we see that 
$\cp{k_{16}}{k_{12}\e_5}$, whence we can remove a further factor
of $k_{16}$ in this congruence.  It therefore follows that 
\[
N_0(k_{12},k_{16})=
\#\left\{\al_2' \mid
\begin{array}{l}
\congr{\al_2'^2\e_1}{-\e_3\e_4^2k_{16}\e_6'^3\e_7}{k_{12}\e_5},\\
h(Y_0, \al_2'k_{16}/Y_2, \e_6/Y_6, \e_7/Y_7) \le 1,\\
\cp{\al_2'}{\e_2\e_3\e_4}
\end{array}
\right\},
\]
since $\cp{k_{16}}{\e_2\e_3\e_4}$.

Note that $\cp{k_{12}\e_5}{\e_1}$ and
$\cp{k_{12}\e_5}{\e_3\e_4^2k_{16}\e_6'^3\e_7}$. It therefore follows
that for each $\al'_2$ satisfying the congruence, there is a unique $1 \le \rho
\le k_{12}\e_5$, with 
\begin{equation}
  \label{eq:anna}
\cp{\rho}{k_{12}\e_5}, \quad
\congr{\rho^2\e_1}{-\e_3\e_6\e_7}{k_{12}\e_5}, \quad
\end{equation}
such that
\[
\congr{\al_2'}{\rho\e_4\e_6'}{k_{12}\e_5}.
\] 
Thus we obtain
\[
N_0(k_{12},k_{16})=
\sum_{\substack{1 \le \rho \le k_{12}\e_5\\\text{\scriptsize{\eqref{eq:anna} holds}}}}
\#\left\{\al_2' \mid
\begin{array}{l}
\congr{\al_2'}{\rho\e_4\e_6'}{k_{12}\e_5},\\
h(Y_0, \al_2'k_{16}/Y_2, \e_6/Y_6, \e_7/Y_7) \le 1,\\
\cp{\al_2'}{\e_2\e_3\e_4}
\end{array}
\right\}.
\]
We remove $\cp{\al_2'}{\e_2\e_3\e_4}$ by a further application of
M\"obius inversion. Writing $\al_2'=k_2\al_2''$, we see that 
$N_0(k_{12},k_{16})$ is equal to 
\begin{equation*}
  \sum_{\substack{1 \le \rho \le k_{12}\e_5\\\text{\scriptsize{\eqref{eq:anna}
          holds}}}}
  \sum_{k_2\mid \e_2\e_3\e_4}\mu(k_2)
  \#\left\{\al_2'' \mid
    \begin{array}{l}
      \congr{k_2\al_2''}{\rho\e_4\e_6'}{k_{12}\e_5},\\
      h(Y_0, \al_2''k_{16}k_2/Y_2, \e_6/Y_6, \e_7/Y_7) \le 1
    \end{array}
  \right\}.
\end{equation*}
The summand vanished unless $\cp{k_2}{k_{12}\e_5}$, since $p\mid k_2,
k_{12}\e_5$ implies $p\mid k_{12}\e_5, \rho\e_4\e_6'$, which is
forbidden. Thus we may restrict our summation over $k_2$ to
$\cp{k_2}{k_{12}\e_5}$, and it therefore follows that the number of available
$\al_2''$ is \[\frac{Y_2}{k_{12}k_{16}k_2\e_5}g_0(Y_0, \e_6/Y_6,
\e_7/Y_7)+O(1),\] where $g_0$ is given by \eqref{eq:g0}. Recall the definition
of the function $\phi^*$ from \S~\ref{s:af}.  We are now ready to establish
the following result.

\begin{lemma}\label{lem:a1a2}
We have
\[\Nzero = \frac{Y_2}{\e_5}g_0(Y_0,
\e_6/Y_6, \e_7/Y_7) \thzero(\ee, \e_6, \e_7) + O(\Rzero(\ee, \e_6,
\e_7;B))\] with
\[
\thzero(\ee,\e_6,\e_7) :=
    \frac{\phis(\e_6)\phis(\e_2\e_3\e_4)}{\phis(\gcd(\e_6,\e_4))}
\hspace{-0.5cm}
\sum_{\substack{k_{12}\mid\e_2\\\cp{k_{12}}{\e_1\e_3\e_4}}} 
\hspace{-0.5cm}
    \frac{\mu(k_{12})}{k_{12}\phis(\gcd(\e_2,k_{12}\e_5))}
\sum_{\substack{1 \le \rho \le
    k_{12}\e_5\\\text{\scriptsize{\eqref{eq:anna} holds}}}}
\hspace{-0.2cm}
1,
\]
and \[\sum_{\ee, \e_6,\e_7} \Rzero(\ee, \e_6, \e_7;B)\ll B(\log B)^3.\] 
\end{lemma}

The final statement in Lemma \ref{lem:a1a2} should be taken to mean that the
overall contribution from the error term in the asymptotic formula for
$\Nzero$, once summed over all of the available $\ee,\e_6,\e_7$, is $ O(B(\log
B)^3)$. What is crucial here is that the exponent of $\log B$ is strictly
smaller than $5$, so that this truly is an acceptable error term from the
point of view of the main theorem. In the case of Lemma \ref{lem:a1a2} we need
to sum $\Rzero(\ee, \e_6, \e_7;B)$ over all $\ee,\e_6,\e_7$ which satisfy the
height conditions \eqref{eq:H0}--\eqref{eq:H4}, and the coprimality conditions
\eqref{eq:cpe6}--\eqref{eq:cpe}. In the arguments to follow there will be
several points at which the overall contribution from various error terms
needs to be estimated. In each case we will not stress the precise conditions
on the variables to be summed over, these being invariably self-evident.

\begin{proof}[Proof of Lemma \ref{lem:a1a2}]
Tracing through our argument above, it follows that 
  \[
\Nzero =  \frac{Y_2}{\e_5}g_0(Y_0,
  \e_6/Y_6, \e_7/Y_7) \thzero(\ee, \e_6, \e_7)+O(\Rzero(\ee, \e_6, \e_7;B)),\] 
with 
\[\thzero =
\sum_{\substack{k_{16}\mid \e_6, k_{12}\mid
    \e_2\\\cp{k_{12}k_{16}}{\e_1\e_3\e_4}}}
\frac{\mu(k_{12})\mu(k_{16})}{k_{12}k_{16}} \sum_{\substack{1 \le \rho \le
    k_{12}\e_5\\\text{\scriptsize{\eqref{eq:anna} holds}}}}
\sum_{\substack{k_2\mid
    \e_2\e_3\e_4\\\cp{k_2}{k_{12}\e_5}}}\frac{\mu(k_2)}{k_2},\] and
\begin{align*}
  \Rzero(\ee, \e_6, \e_7;B) &\ll 2^{\omega(\e_2\e_3\e_4)+\omega(\eta_6)}
  \sum_{k_{12}\mid \e_2} |\mu(k_{12})| \sum_{\substack{1 \le \rho \le
      k_{12}\e_5\\\text{\scriptsize{\eqref{eq:anna}
          holds}}}} 1\\
  &\ll 2^{\omega(\e_2)+\omega(\e_2\e_5)} 2^{\omega(\e_2\e_3\e_4)+\omega(\eta_6)} \\
  &\leq
  8^{\omega(\e_2)}2^{\omega(\e_3)+\omega(\e_4)+\omega(\e_5)+\omega(\e_6)}.
\end{align*}
We have used here the fact that the congruence in \eqref{eq:anna} has at most
$2^{\omega(k_{12}\e_5)}\leq 2^{\omega(\e_2\e_5)}$ solutions $\rho$ modulo
$k_{12}\e_5$.

On noting that $\cp{\e_6}{\e_1\e_3}$ and
$\cp{\e_3\e_4}{k_{12}\e_5}$, we deduce that 
\begin{align*}
\thzero
&=
\sum_{\substack{k_{12}\mid \e_2\\\cp{k_{12}}{\e_1\e_3\e_4}}}
\frac{\mu(k_{12})}{k_{12}} \frac{\phi^*(\e_6)}{\phi^*(\gcd(\e_6,\e_4))}
\frac{\phi^*(\e_2\e_3\e_4)}{\phi^*(\gcd(\e_2,k_{12}\e_5))}
\sum_{\substack{1 \le \rho \le k_{12}\e_5\\\text{\scriptsize{\eqref{eq:anna} holds}}}}1.
\end{align*}
This completes the proof of the main term in the lemma.

To show that $\Rzero(\ee, \e_6, \e_7;B)$ makes a satisfactory contribution
once it is summed over all $\ee,\e_6,\e_7$ satisfying the height conditions in
\eqref{eq:anna-clap}, we begin by summing over $\e_7$. Thus it follows that
\[
  \begin{split}
    \sum_{\ee,\e_6,\e_7}\Rzero(\ee, \e_6, \e_7;B) 
    &\ll \sum_{\ee, \e_6} \frac
    {8^{\omega(\e_2)}2^{\omega(\e_3)+\omega(\e_4)+\omega(\e_5)+\omega(\e_6)}B}
    {\fbase 1 2 2 2 1\e_6^2}\\ 
&\ll B(\log B)^3,
  \end{split}
  \]
as required to complete the proof of the lemma.
\end{proof}

\subsection{Estimating $N_a(B)$ --- second  step}\label{s:a}

In this section our task is to sum the main term in Lemma
\ref{lem:a1a2} over all of the relevant $\e_6$ and $\eta_7$, such that
\eqref{eq:Ta} holds. As we've already indicated, we will begin by
summing over the $\e_6$. 
For fixed $\ee, \e_7$ satisfying the coprimality conditions
\eqref{eq:cpe7} and  \eqref{eq:cpe}, define $\Nonea:=\Nonea(\ee,
\e_7;B)$ to be the sum of the main term in Lemma~\ref{lem:a1a2} over
all $\e_6\in \ZZp$ such that  
the coprimality condition \eqref{eq:cpe6} holds, and furthermore, 
$\e_6 \ge |\e_7|$.

We begin by noting that it is possible to remove $\e_5$ from \eqref{eq:cpe6},
replacing this coprimality condition by $\cp{\e_6}{\e_1\e_2\e_3}$. Indeed, if
$p\mid \e_6,\e_5$ then \eqref{eq:anna} implies that we must have $p\mid
\rho^2\e_1$, which is forbidden.  Since $\cp{\e_3\e_7}{k_{12}\e_5}$, so there
exists a unique integer $\beta\in [1, k_{12}\e_5]$ such
that \[\congr{\rho^2\e_1}{-\e_3\e_7\beta}{k_{12}\e_5}.\] It therefore follows
that
\[
\Nonea = \frac{Y_2}{\e_5}\phis(\e_2\e_3\e_4)
\sum_{\substack{k_{12}\mid \e_2\\\cp{k_{12}}{\e_1\e_3\e_4}}}
\frac{\mu(k_{12})}{k_{12}\phis(\gcd(\e_2,k_{12}\e_5))} \sum_{\substack{1 \le
    \rho \le k_{12}\e_5\\\gcd(\rho,k_{12}\e_5)=1}} A,
\] 
where 
\begin{equation*}
  \begin{split}
    A&=\sum_{\substack{\e_6 \in \ZZp \\
        \e_6 \ge |\e_7|\\ \congr {\e_6}\beta{k_{12}\e_5}}}
    f_{\e_4,\e_1\e_2\e_3}(\e_6)g_0(Y_0,\e_6/Y_6,\e_7/Y_7).
  \end{split}
\end{equation*}
Here $f_{\e_4,\e_1\e_2\e_3}$ is given by \eqref{eq:def_f}.  
Since $g_0(Y_0,\e_6/Y_6,\e_7/Y_7)=0$ for $\e_6 > B$, we may restrict the
summation to $\e_6$ in the range $|\e_7| \le \e_6 \le B$.

We will estimate $A$ using 
Lemma~\ref{lem:summation_f_g}. This produces a main term and an error
term, the latter having size 
\[\ll 2^{\omega(\e_1\e_2\e_3)}(\log B) \sup_{t_6}g_0(Y_0,t_6,\e_7/Y_7),\]
where the supremum is over all $t_6 \in \RR$ such that 
$ Y_6t_6\ge |\e_7|$. 
This therefore gives an overall contribution
\begin{equation}
  \label{eq:overall_anna}
  \ll Y_2 2^{\omega(\e_2)+\omega(\e_1\e_2\e_3)}(\log B)
  \sup_{t_6}g_0(Y_0,t_6,\e_7/Y_7),
\end{equation}
to $\Nonea$, since \[\frac{1}{k_{12}\e_5}\sum_{\substack{1 \le \rho \le
    k_{12}\e_5\\\gcd(\rho,k_{12}\e_5)=1}}1=\phi^*(k_{12}\e_5).\] The main term
in our application of Lemma~\ref{lem:summation_f_g} to $A$ is simply
\[\Theta(\ee, k_{12})\frac{Y_6}{k_{12}\e_5}\int_{
  Y_6t_6\ge |\e_7|, t_6>0} g_0(Y_0,t_6,\e_7/Y_7)\dd t_6,\] with
\begin{align*}
  \Theta(\ee, k_{12}) &=
  \frac{\phis(\e_1\e_2\e_3)}{\zeta(2)\phis(\gcd(\e_1\e_2\e_3, k_{12}\e_5))}
  \prod_{p\mid \e_1\e_2\e_3\e_4\e_5}\Big(1-\frac{1}{p^2}\Big)^{-1}\\
  &=\frac{\phis(\e_1\e_2\e_3)}{\zeta(2)\phis(\gcd(\e_2, k_{12}\e_5))}
  \prod_{p\mid \e_1\e_2\e_3\e_4\e_5}\Big(1-\frac{1}{p^2}\Big)^{-1}.
\end{align*}
Here we have used the fact that $\cp{\e_1\e_3}{k_{12}\e_5}$. Note for
future reference that $\Theta(\ee, k_{12}) \ll 1$.
We are now ready to establish the following result.

\begin{lemma}\label{lem:a_e6}
We have
\[\Nonea=\frac{Y_2Y_6}{\e_5}\gonea(Y_0,\e_7/Y_7, \ee;  B)
\thonea(\ee) +O(\Ronea(\ee, \e_7;B))\]
with 
\[\thonea(\ee):=\sum_{\substack{k_{12}\mid \e_2\\\cp{k_{12}}{\e_1\e_3\e_4}}}
\frac{\mu(k_{12})}{k_{12}} \Theta(\ee,k_{12})\phis(\e_2\e_3\e_4\e_5),\]
and \[\sum_{\ee, \e_7} \Ronea(\ee, \e_7;B)\ll B(\log B)^3.\]
\end{lemma}

\begin{proof}
It is clear from our calculations above that the main term in our
estimate for $\Nonea$ is equal to 
$Y_2Y_6 \gonea(Y_0,\e_7/Y_7, \ee;  B)\thonea(\ee)/\e_5$, with 
\begin{align*}
\thonea(\ee)&=
\sum_{\substack{k_{12}\mid \e_2\\\cp{k_{12}}{\e_1\e_3\e_4}}}
\frac{\mu(k_{12})\phis(\e_2\e_3\e_4)
\phis(k_{12}\e_5)}{k_{12}\phis(\gcd(\e_2\e_3\e_4,k_{12}\e_5))} 
\Theta(\ee, k_{12})\\
&=  \phis(\e_2\e_3\e_4\e_5) 
\sum_{\substack{k_{12}\mid \e_2\\\cp{k_{12}}{\e_1\e_3\e_4}}}
\frac{\mu(k_{12})}{k_{12}}\Theta(\ee, k_{12}),
\end{align*}
since $k_{12}\e_5$ is coprime to $\e_3\e_4$ and
every divisor of $k_{12}$ divides $\e_2$. This completes the proof of
the main term in the lemma. 

Turning to the overall contribution from the error term $\Ronea(\ee, \e_7;B)$,
which we have already seen has size \eqref{eq:overall_anna}, we conclude from
\eqref{eq:anna-clap} and \eqref{eq:Ta} that \[|\fbase 2 4 3 2 3 \e_7| \leq
B,\] for the $\ee, \e_7$ that we need to sum over.  We therefore deduce from
Lemma~\ref{lem:bound_g}\eqref{it:bound_gzero} that
\begin{align*}
  \sum_{\ee, \e_7} \Ronea(\ee, \e_7;B) &\ll \log B \sum_{\ee, \e_7} Y_2
  4^{\omega(\e_2)}2^{\omega(\e_1)+\omega(\e_3)} \cdot
  \frac{Y_7^{1/2}}{Y_0|\e_7|^{1/2}}\\
  &= \log B\sum_{\ee, \e_7}
  \frac{4^{\omega(\e_2)}2^{\omega(\e_1)+\omega(\e_3)}B^{1/2}}
  {\fbase{1/2}{0}{0}{0}{-1/2}|\e_7|^{1/2}}\\
  &\ll \log B \sum_{\ee} \frac{4^{\omega(\e_2)}2^{\omega(\e_1)+\omega(\e_3)} B
  }{\fbase{3/2}{2}{3/2}{1}{1}}\\
  &\ll B (\log B)^3,
\end{align*}
as required to complete the proof of the lemma.
\end{proof}

Lemma \ref{lem:a_e6} takes care of the summation of the main term in
Lemma \ref{lem:a1a2} over all of the relevant $\e_6$. We proceed to
sum the resulting main term over the $\e_7$.
Thus we let 
\[
\Ntwoa:=\Ntwoa(\ee;B) = \sum_{\substack{\e_7 \in \ZZnz\\\text{\eqref{eq:cpe7} holds}}}
\frac{Y_2Y_6}{\e_5}
\gonea(Y_0,\e_7/Y_7; \ee; B) \thonea(\ee)
\]
We begin with an application of M\"obius inversion to remove the coprimality
condition \eqref{eq:cpe7}. This gives
\begin{align*}
    \Ntwoa &= \frac{Y_2Y_6}{\e_5}\thonea(\ee)
    \sum_{k_7\mid \e_1\e_2\e_3\e_4\e_5}\mu(k_7)\sum_{|\e_7'|\ge 1}
    \gonea(Y_0,k_7\e_7'/Y_7; \ee;  B),
\end{align*}
where we have written $\e_7=k_7\e_7'$. Partial summation now yields 
\begin{align*}
  \Ntwoa =& \frac{Y_2Y_6Y_7}{\e_5}\thonea(\ee)\sum_{k_7\mid
    \e_1\e_2\e_3\e_4\e_5} \frac{\mu(k_7)}{k_7} \int_{|t_7| \geq k_7/Y_7}
  \gonea(Y_0,t_7;\ee;B) \dd t_7 \\&\quad+O\Big(\frac{Y_2Y_6}{\e_5}
  |\thonea(\ee)|\sum_{k_7 \mid \e_1\e_2\e_3\e_4\e_5} |\mu(k_7)| \sup_{|t_7|
    \ge k_7/Y_7} \gonea(Y_0,t_7;\ee;B)\Big).
\end{align*}
The following result constitutes the final outcome of our summation
over $\e_7$.

\begin{lemma}\label{lem:a_e7}
  We have \[\Ntwoa = \frac{Y_2Y_6Y_7}{\e_5} \gtwoa(Y_0,\ee; B)\thtwoa(\ee)
  +O(\Rtwoa(\ee;B))\] with \[\thtwoa(\ee) :=
  \phis(\e_1\e_2\e_3\e_4\e_5)\thonea(\ee),\] and \[\sum_{\ee} \Rtwoa(\ee;B)\ll
  B(\log B)^{5-2/7}.\]
\end{lemma}

\begin{proof}
The effect of replacing the integral $\int_{|t_7| \geq k_7/Y_7}
\gonea(Y_0,t_7;\ee;B) \dd t_7$ by $\gtwoa(Y_0, \ee; B)$ in our estimate for
$\Ntwoa$, is to create an additional term
\[\frac{Y_2Y_6Y_7}{\e_5} |\thonea(\ee)|
\sum_{k_7\mid \e_1\e_2\e_3\e_4\e_5} \frac{|\mu(k_7)|}{k_7} \int_{1/Y_7 <
  |t_7| < k_7/Y_7} \gonea(Y_0,t_7;\ee;B) \dd t_7,\] that must become part of
$\Rtwoa(\ee;B)$.  Let us think of this as the first term in $\Rtwoa(\ee;B)$.
The second term that appears in $\Rtwoa(\ee;B)$ is the error appearing in
the asymptotic formula for $\Ntwoa$ that directly precedes the statement of
the lemma.

We will need to estimate the overall contribution from both of these
terms separately. It will be convenient to note that 
$\thonea(\ee)=O(\phid(\e_2))$, in the notation of \S~\ref{s:af}.
Let $\lambda>0$ be a parameter to be selected in due course. 
Our argument will depend upon whether or not  $\fbase 3 6 4 2 5 < \lambda B$
in the summation over the $\ee$. Accordingly let $E_1(\lambda)$ denote
the overall contribution from the two errors terms once summed over 
$\ee$ such that 
\begin{equation}
  \label{eq:case1}
  \fbase 3 6 4 2 5 < \lambda B,
\end{equation}
and let $E_2(\lambda)$ denote the remaining contribution from $\ee$
such that
\begin{equation}
  \label{eq:case2}
  \fbase 3 6 4 2 5 \geq  \lambda B.
\end{equation}

Beginning with the estimation of $E_1(\lambda)$, we employ
Lemma~\ref{lem:bound_g}\eqref{it:bound_gonea} to conclude that 
\begin{align*}
\int_{1/Y_7}^{k_7/Y_7} \gonea(Y_0,t_7;\ee;B) \dd t_7 \ll
  \int_{1/Y_7}^{k_7/Y_7} \frac{1}{Y_0|t_7|^{7/6}} \dd t_7 
&\ll
  \frac{Y_7^{1/6}}{Y_0}.
\end{align*}
Once summed over all $\ee$ such that \eqref{eq:case1} holds, we use
\eqref{eq:phid} to estimate the overall contribution from the first term in
$\Rtwoa(\ee;B)$ as
\[
  \begin{split}
    &\ll \sum_{\ee} \sum_{k_7\mid\e_1\e_2\e_3\e_4\e_5} \frac{|\mu(k_7)|}{k_7}
    \frac{\phid(\e_2) Y_2Y_6Y_7^{7/6}}{\e_5Y_0}\\
    &\ll \sum_{\ee} \frac{\phid(\e_1\e_2\e_3\e_4\e_5)\phid(\e_2)
      B^{5/6}}{\fbase{1/2}{0}{1/3}{2/3}{1/6}}\\
    &\ll \sum_{\e_1,\e_2,\e_3,\e_4} \frac{
      \phid(\e_1)\phid(\e_3)\phid(\e_4)\phid(\e_2)^2
      \lambda^{1/6} B}{\fbase 1 1 1 1 0}\\
    &\ll \lambda^{1/6} B (\log B)^4.
  \end{split}
\]
Turning to the overall contribution from the second
term in $\Rtwoa(\ee;B)$, we again deduce from 
Lemma~\ref{lem:bound_g}\eqref{it:bound_gonea} that 
\[
\sup_{|t_7| \ge k_7/Y_7}\gonea(Y_0,t_7;\ee;B) \ll \sup_{|t_7| \ge k_7/Y_7}
  \frac{1}{Y_0|t_7|^{7/6}} \ll \frac{Y_7^{7/6}}{Y_0k_7^{7/6}}.
\] 
Hence, in this case too, we obtain the overall contribution 
  \[
\ll \sum_{\ee} \frac{\phid(\e_2)Y_2Y_6Y_7^{7/6}}{\e_5Y_0}
    \ll \sum_{\e_1, \e_2, \e_3, \e_4}
    \frac{\phid(\e_2)\lambda^{1/6}B}{\fbase{1}{1}{1}{1}{0}}
    \ll \lambda^{1/6}B(\log B)^4.
  \] 
Thus far we have shown that $E_1(\lambda)\ll \lambda^{1/6}B(\log
B)^4$.

It remains to produce a suitable upper bound for $E_2(\lambda)$. 
It will be convenient to record the estimates
\[\sum_{n\leq x} \frac{2^{\omega(n)}\phid(n)}{n} \ll (\log x)^2, 
\quad \sum_{n>x} \frac{h^{\omega(n)}}{n^a} \ll x^{1-a}(\log x)^{h-1}.\]
The second inequality 
is valid for any $h \in \ZZp$ and any $a>1$, and 
follows on combining partial summation with the
bound
\[\sum_{n\leq x} h^{\omega(n)}
\leq \sum_{n\leq x} \sum_{n=d_1\ldots d_h}1 
\ll \sum_{d_1,\ldots,d_{h-1}\leq x} \frac{x}{d_1\ldots d_{h-1}}\ll x
(\log x)^{h-1}.\]
Beginning with the first term in $\Rtwoa(\ee;B)$, 
we deduce from Lemma~\ref{lem:bound_g}\eqref{it:bound_gonea} that 
\[
\int_{1/Y_7}^{k_7/Y_7}
  \gonea(Y_0,t_7;\ee;B) \dd t_7 \ll \int_{1/Y_7}^{k_7/Y_7} \frac{1}{Y_0^8}
  \dd t_7 \ll \frac{k_7}{Y_7Y_0^8}.
\]
Summing over $\ee$ such that \eqref{eq:case2} holds, we therefore
obtain the overall contribution 
  \[
  \begin{split}
&\ll \sum_{\ee} \sum_{k_7\mid\e_1\e_2\e_3\e_4\e_5}
    \frac{|\mu(k_7)|}{k_7}\frac{\phid(\e_2)k_7Y_2Y_6}{\e_5Y_0^8}\\
    &\ll \sum_{\ee}
    \frac{2^{\omega(\e_1\e_2\e_3\e_4\e_5)}\phid(\e_2)B^2}{\fbase 4 7 5 3 6}\\
    &\ll \sum_{\e_1,\dots, \e_4}
    \frac{2^{\omega(\e_1\e_2\e_3\e_4)}\phid(\e_2)B \log B}{\lambda \fbase 1 1
      1 1 0}\\
    &\ll \lambda^{-1} B(\log B)^9,
  \end{split}\]
by \eqref{eq:phid}.
Similarly, for the contribution from the second term in $\Rtwoa(\ee;B)$,
we may use Lemma~\ref{lem:bound_g}\eqref{it:bound_gonea} to deduce
the overall contribution 
  \[
  \begin{split}
\ll \sum_{\ee} \frac{2^{\omega(\e_1\e_2\e_3\e_4\e_5)}\phid(\e_2)Y_2Y_6}{\e_5Y_0^8}
\ll \lambda^{-1} B(\log B)^9.
\end{split}\] Taken together this shows that $ E_2(\lambda)\ll \lambda^{-1}
B(\log B)^9$.  We choose $\lambda = (\log B)^{30/7}$, which therefore gives
the overall contribution \[\sum_{\ee} \Rtwoa(\ee;B)\ll B(\log B)^{33/7},\] as
required.
\end{proof}

\subsection{Estimating $N_{b}(B)$ --- second step}\label{s:b}

We must now return to the main term in Lemma \ref{lem:a1a2}, but this
time reverse the order of summation for $\e_6$ and $\e_7$. 
This will allow us to make use of the inequality \eqref{eq:Tb} in our
treatment of the error terms. 
We begin with the summation over $\e_7$.
For fixed $\ee, \e_6$ satisfying the coprimality conditions
\eqref{eq:cpe6} and  \eqref{eq:cpe}, define $\Noneb:=\Noneb(\ee,
\e_6;B)$ to be the sum of the main term in Lemma
\ref{lem:a1a2} over all $\e_7\in \ZZnz$ such that 
the coprimality condition \eqref{eq:cpe7} holds, and furthermore, 
$|\e_7| > \e_6=\max\{\e_6,1\}$.

Our argument is very similar in spirit to the preceding
section. Removing \eqref{eq:cpe7} with an application of M\"obius
inversion, we find that 
\[
\begin{split}
  \Noneb =
  &\frac{Y_2}{\e_5}\frac{\phis(\e_6)\phis(\e_2\e_3\e_4)}{\phis(\gcd(\e_6,\e_4))}
\sum_{\substack{k_{12}\mid \e_2\\\cp{k_{12}}{\e_1\e_3\e_4}}}
  \frac{\mu(k_{12})}{k_{12}\phis(\gcd(\e_2,k_{12}\e_5))} \\
  &\quad \times\sum_{\substack{1 \le \rho \le
    k_{12}\e_5\\\gcd(\rho,k_{12}\e_5)=1}}
  \sum_{\substack{k_7\mid \e_1\e_2\e_3\e_4\e_5\\\cp{k_7}{k_{12}\e_5}}}\mu(k_7)
  A, 
\end{split}
\] 
where 
\[
A = \sum_{\substack{\e_7' \in
    \ZZnz\\\congr{\rho^2\e_1}{-\e_3\e_6k_7\e_7'}{k_{12}\e_5}\\
k_7|\e_7'| > \e_6}}
g_0(Y_0,\e_6/Y_6,k_7\e_7'/Y_7),
\] 
and we have written $\e_7 = k_7\e_7'$.
Note that we have been able to add the
constraint $\cp{k_7}{k_{12}\e_5}$ in the sum over $k_7$, since  $A=0$ otherwise.

Since $\cp{\e_3\e_6k_7}{k_{12}\e_5}$, it follows from an easy
application of partial summation that  
\[
A = \frac{Y_7}{k_{12}k_7\e_5}
\goneb(Y_0,\e_6/Y_6;\ee;B) + O\Big(\sup_{t_7}g_0(Y_0,\e_6/Y_6,t_7)\Big),
\]
where the supremum is over $t_7\in \RR$ such that $|t_7| >\e_6/Y_7$.
We may now establish the following result.

\begin{lemma}\label{lem:b_e7}
We have
\[
\Noneb =  \frac{Y_2Y_7}{\e_5}\goneb(Y_0, \e_6/Y_6; \ee;B)
  \thoneb(\ee)\frac{\phis(\e_6)}{\phis(\gcd(\e_6,\e_4))}
 +O(\Roneb(\ee, \e_6;B))
\]
with 
\[
\thoneb(\ee) :=
\phis(\e_2\e_3\e_4\e_5)\phis(\e_1\e_2\e_3\e_4)\sum_{\substack{k_{12}\mid
    \e_2\\\cp{k_{12}}{\e_1\e_3\e_4}}}\frac{\mu(k_{12})}
{k_{12}\phis(\gcd(\e_2,k_{12}\e_5))},
\]
and 
\[\sum_{\ee, \e_6} \Roneb(\ee, \e_6;B)\ll B(\log B)^3.\]
\end{lemma}

\begin{proof}
It is clear that the main term in the lemma is valid with 
\begin{align*}
  \thoneb &=\sum_{\substack{k_{12}\mid \e_2\\\cp{k_{12}}{\e_1\e_3\e_4}}}
  \frac{\mu(k_{12})\phis(\e_2\e_3\e_4)\phis(k_{12}\e_5)}
  {k_{12}\phis(\gcd(\e_2,k_{12}\e_5))}
  \frac{\phis(\e_1\e_2\e_3\e_4)}{\phis(\gcd(\e_1\e_2\e_3\e_4,k_{12}\e_5))}\\
  &=\sum_{\substack{k_{12}\mid \e_2\\\cp{k_{12}}{\e_1\e_3\e_4}}}
  \frac{\mu(k_{12})}{k_{12}} \phis(k_{12}\e_2\e_3\e_4\e_5)
  \frac{\phis(\e_1\e_2\e_3\e_4)}{\phis(\gcd(\e_2,k_{12}\e_5))}\\
  &=\phis(\e_2\e_3\e_4\e_5)\phis(\e_1\e_2\e_3\e_4)\sum_{\substack{k_{12}\mid
      \e_2\\\cp{k_{12}}{\e_1\e_3\e_4}}}\frac{\mu(k_{12})}
  {k_{12}\phis(\gcd(\e_2,k_{12}\e_5))},
  \end{align*}
as claimed. We have used here the fact that $\cp{k_{12}\e_5}{\e_1\e_3\e_4}$.

For the error term, we deduce from \eqref{eq:anna-clap} that $ \fbase 2 4 3 2
3 \e_6 \le B, $ for the $\ee, \e_6$ that we need to sum $\Roneb(\ee, \e_6;B)$
over.  Using Lemma~\ref{lem:bound_g}\eqref{it:bound_gzero} to bound $g_0$, we
easily deduce that
\begin{align*}
  \sum_{\ee, \e_6} \Roneb(\ee, \e_6;B) &\ll \sum_{\ee,\e_6}Y_2
  2^{\omega(\e_2)+\omega(\e_1\e_2\e_3\e_4)}\sup_{|t_7| >
    \e_6/Y_7} g_0(Y_0,\e_6/Y_6,t_7)\\
  &\ll \sum_{\ee,\e_6}
  \frac{2^{\omega(\e_2)+\omega(\e_1\e_2\e_3\e_4)}Y_2Y_7^{1/2}}{Y_0\e_6^{1/2}}\\
  &=\sum_{\ee,\e_6}\frac{2^{\omega(\e_2)+\omega(\e_1\e_2\e_3\e_4)}B^{1/2}}
  {\fbase{1/2}0 0 0{-1/2}\e_6^{1/2}}\\
  &\ll
  \sum_{\ee}\frac{2^{\omega(\e_2)+\omega(\e_1\e_2\e_3\e_4)}B}
  {\fbase{3/2}{2}{3/2}{1}{1}}\\
  &\ll B\sum_{\e_4, \e_5}\frac{2^{\omega(\e_4)}}{\e_4\e_5}.
\end{align*}
But this is $O(B (\log B)^3)$, as required.  This completes the proof of the
lemma.
\end{proof}

We must now sum the main term in Lemma \ref{lem:b_e7} over all of the relevant
$\e_6$, and then over $\e_1,\ldots,\e_5$. In doing so it will be convenient
distinguish between values of $\ee, \e_6$ such that
\begin{equation}\label{eq:cond}
\fbase{2}{4}{3}{2}{3}\leq \frac{B}{(\log B)^A},
\end{equation}
for some $A>0$, and those for which this inequality does not hold. We write
$N_{b_1}(B;A)$ and $N_{b_2}(B;A)$ for the corresponding contributions.
The following result shows that $N_{b_2}(B;A)$
makes a negligible contribution to $N_{U,H}(B)$.

\begin{lemma}\label{lem:b2}
We have $N_{b_2}(B;A) \ll_A B (\log B)^4(\log\log B)$.
\end{lemma}

\begin{proof}
Once taken in conjunction with the inequalities for $\ee, \e_6$ in
\eqref{eq:anna-clap}, the failure of \eqref{eq:cond} 
clearly implies that we must sum over
$\ee,\e_6$ for which 
\begin{equation}
  \label{eq:height'}
\eta_1^2\eta_2^4\eta_3^3\eta_4^2\eta_5^3\eta_6
\leq B, \quad \eta_6<(\log B)^A.
\end{equation}
Recalling  the definition of the main term from Lemma \ref{lem:b_e7}, we
see  that 
\begin{align*}
N_{b_2}(B;A)
&\ll \sum_{\substack{\ee, \e_6\\ 
\text{\scriptsize{\eqref{eq:height'} holds}}
}}
\frac{Y_2Y_7}{\e_5}\goneb(Y_0, \e_6/Y_6; \ee;B)
  \thoneb(\ee)\frac{\phis(\e_6)}{\phis(\gcd(\e_6,\e_4))}\\
&\ll \sum_{\substack{\ee, \e_6\\ 
\text{\scriptsize{\eqref{eq:height'} holds}}
}}
\frac{Y_2Y_6^{3/4}Y_7 \phid(\eta_2)}{Y_0 \e_5\e_6^{3/4}},
\end{align*}
using Lemma \ref{lem:bound_g}(\ref{it:bound_goneb}).
In view of the definitions of the $Y_i$ we conclude that
\begin{align*}
N_{b_2}(B;A)
&\ll B^{3/4}\sum_{\substack{\ee, \e_6\\ 
\text{\scriptsize{\eqref{eq:height'} holds}}
}}
\frac{\phid(\eta_2)}{
\fbase{1/2}{0}{1/4}{1/2}{1/4}\e_6^{3/4}}\\
&\ll B \sum_{\substack{\e_2,\ldots, \e_6\\ 
\text{\scriptsize{\eqref{eq:height'} holds}}
}}
\frac{\phid(\eta_2)}{
\fbase{0}{1}{1}{1}{1}\e_6}.
\end{align*}
This last expression is clearly satisfactory for the lemma 
by \eqref{eq:phid} with $j=1$ and the fact that the $\e_6$ summation
is over $\e_6<(\log B)^A$.
\end{proof}

Our focus now shifts to estimating $N_{b_1}(B;A)$, deemed to be the overall
contribution from the main term in Lemma \ref{lem:b_e7} that arises from
$\ee,\e_6$ for which \eqref{eq:cond} holds.  For the moment let
$\Ntwob:=\Ntwob(\ee;B)$ be the quantity obtained by summing the main term in
Lemma \ref{lem:b_e7}'s estimate for $\Noneb(\ee,\e_6;B)$, over all $\e_6\in
\ZZp$ such that \eqref{eq:cpe6} holds.  An application of
Lemma~\ref{lem:summation_f_g} with $\alpha=0$ and $q=1$ therefore reveals that
\begin{align*}
  \Ntwob =& \frac{Y_2Y_7}{\e_5}\thoneb(\ee) \sum_{\substack{\e_6 \ge 1
}} f_{\e_4,\e_1\e_2\e_3\e_5}(\e_6)\goneb(Y_0,\e_6/Y_6;\ee;B)\\
  =&\frac{Y_2Y_6Y_7}{\e_5} \gtwob(Y_0;\ee;B)
  \thoneb(\ee)\frac{\phi^*(\e_1\e_2\e_3\e_5)}{\zeta(2)}
\prod_{p\mid \e_1\e_2\e_3\e_4\e_5} 
\Big(1-\frac{1}{p^2}\Big)^{-1}\\
  &\quad +O\left(\frac{Y_2Y_7}{\e_5}|\thoneb(\ee)|(\log
    B)2^{\omega(\e_1\e_2\e_3\e_5)}\sup_{t_6}\goneb(Y_0,t_6;\ee;B)\right)\\
  &\quad +O\left(\frac{Y_2Y_6Y_7}{\e_5} |\thoneb(\ee)| \int_{
0 < t_6 < 1/Y_6}
\goneb(Y_0,t_6;\ee;B)\dd t_6\right),
\end{align*}
where the supremum is over all 
$t_6 \ge 1/Y_6$.  The following result is now straightforward.

\begin{lemma}\label{lem:b_e6}
  We have 
\[\Ntwob = \frac{Y_2Y_6Y_7}{\e_5}\gtwob(Y_0;
  \ee;B)\thtwob(\ee)+O(\Rtwob(\ee;B))
\]
with 
\[
\thtwob(\ee) :=
\thoneb(\ee)\frac{\phi^*(\e_1\e_2\e_3\e_5)}{\zeta(2)}\prod_{p\mid
  \e_1\e_2\e_3\e_4\e_5}
\Big(1-\frac{1}{p^2}\Big)^{-1},
\]
and 
\[\sum_{\substack{\ee\\\text{\eqref{eq:cond} holds}}} \Rtwob(\ee;B)\ll B(\log
B)^{9-A/4}.\] 
\end{lemma}

\begin{proof}
The value of $\thtwob(\ee)$ in the main term for $\Ntwob$
is a direct consequence of our manipulations above. In considering the
overall contribution from the error term it will be convenient to note that
$\thtwob(\ee)\ll \thoneb(\ee)\ll \phid(\e_2)$.

Once again the error $\Rtwob(\ee;B)$ is comprised of two basic terms, the
first one involving a supremum of $\goneb$ over $t_6$ in an appropriate range,
and the second involving an integration of $\goneb$.  We begin with dealing
with the first term. It is here that we will make critical use of the
inequality \eqref{eq:cond}, that underpins our definition of $N_{b_1}(B;A)$.
The first term in $\Rtwob(\ee;B)$ clearly makes an overall contribution of
\[\ll\sum_{\ee}\frac{2^{\omega(\e_1\e_2\e_3\e_5)}\phid(\e_2)Y_2Y_7 \log
  B}{\e_5}\sup_{t_6 \ge 1/Y_6}\goneb(Y_0,t_6,\ee;B),\]
where the summation is restricted to $\ee$ for which \eqref{eq:cond} holds.
Using Lemma~\ref{lem:bound_g}\eqref{it:bound_goneb} to estimate $\goneb$, we
may bound this as
  \begin{align*}
&\ll \sum_{\ee}
    \frac{2^{\omega(\e_1\e_2\e_3\e_5)}\phid(\e_2)Y_2Y_7Y_6^{3/4}\log B}{\e_5Y_0}\\
    &= \sum_{\ee} \frac{2^{\omega(\e_1\e_2\e_3\e_5)}\phid(\e_2)B^{3/4} \log
      B}{\fbase{1/2}{0}{1/4}{1/2}{1/4}}\\
    &\ll (\log B)^{1-A/4}\sum_{\e_1, \e_2, \e_3, \e_5}
    \frac{2^{\omega(\e_1\e_2\e_3\e_5)}\phid(\e_2)B}{\fbase 1 1 1 0 1 }\\
    &\ll B(\log B)^{9-A/4}.
  \end{align*}
Turning to the contribution from the second term in 
$\Rtwob(\ee;B)$,  we employ Lemma~\ref{lem:bound_g}\eqref{it:bound_goneb}
 and \eqref{eq:cond} to derive the overall contribution 
 \[
  \begin{split}
\ll
    \sum_{\ee} \frac{\phid(\e_2)Y_2Y_6Y_7}{\e_5}
    \int_{0}^{1/Y_6}\frac{1}{Y_0t_6^{3/4}} \dd t_6
    &\ll\sum_{\ee} \frac{\phid(\e_2)Y_2Y_6^{3/4}Y_7}{\e_5Y_0}\\
    &= \sum_{\ee} \frac{\phid(\e_2)B^{3/4}}{\fbase{1/2}{0}{1/4}{1/2}{1/4}}\\
    &\ll (\log B)^{-A/4}\sum_{\e_1,\e_2,\e_3,\e_5} \frac{\phid(\e_2)B}{\fbase
      1 1 1 0 1}\\
    &\ll B(\log B)^{4-A/4}.
  \end{split}
  \]
Together these two upper bounds complete the proof of the lemma. 
\end{proof}

\subsection{The final step}\label{s:final}

Let us take a moment to compile our work so far. We saw at the start of
\S~\ref{s:prelim} that \[N_{U,H}(B)= N_a(B)+ N_b(B).\] It will be convenient
to set $B_0=B/(\log B)^{36}$ in what follows.

The union of Lemmas \ref{lem:a1a2}, \ref{lem:a_e6} and \ref{lem:a_e7}
shows that 
\[N_a(B)=\sum_{\ee\in\EE(B)}
\frac{Y_2Y_6Y_7}{\e_5} \thtwoa(\ee) \gtwoa(Y_0,
  \ee;  B)
+O\big(B(\log B)^{5-2/7}\big),\]
where $\thtwoa(\ee)$ is as in the statement of Lemma \ref{lem:a_e7},
and 
\[\EE(B):=\big\{\ee\in \ZZp^5: \text{\eqref{eq:cpe} holds and $\fbase 2 4 3 2
  3 \leq B$}\big\}.\]
Similarly, we can combine Lemmas \ref{lem:a1a2}, \ref{lem:b_e7}, \ref{lem:b2}
and \ref{lem:b_e6}, taking $A=36$ in the latter two results, to deduce that
\[N_b(B)= \sum_{\substack{\ee\in\EE(B)\\\fbase 2 4 3 2 3 \le B_0}}
\frac{Y_2Y_6Y_7}{\e_5}\thtwob(\ee)\gtwob(Y_0; \ee;B) +O\big(B(\log
B)^4(\log\log B)\big),\]
where $\thtwob(\ee)$ is as in the statement of Lemma \ref{lem:b_e6}.

We would now like to remove the constraint that $\fbase 2 4 3 2 3 \le B_0$ in
our estimate for $N_{b}(B)$.  In view of the fact that $\thtwob(\ee)\ll
\phid(\e_2)$, it easily follows from \eqref{eq:g2l} and
Lemma~\ref{lem:bound_g}\eqref{it:bound_goneb} that
\begin{align*}
\sum_{\substack{\ee\in\EE(B)\\B_0<\fbase 2 4 3 2 3 \le B}}
\hspace{-0.2cm}
\frac{Y_2Y_6Y_7}{\e_5}\thtwob(\ee)\gtwob(Y_0;
  \ee;B)
&\ll 
\sum_{\ee}
\frac{Y_2Y_6Y_7\phid(\e_2)}{\e_5}
\int_{0}^{1/Y_0^4}
\frac{\dd t_6}{Y_0t_6^{3/4}} \\
&\ll 
\sum_{\ee}
\frac{Y_2Y_6^3Y_7\phid(\e_2)}{\e_5}\\
&=
\sum_{\ee}
\frac{B\phid(\e_2)}{\fbase 1 1 1 1 1}\\
&\ll B(\log B)^4(\log\log B).
\end{align*}
In deducing the first bound we have used the fact that $\goneb(t_0,t_6;\ee;
B)=0$
unless $0<t_6\leq 1/t_0^4$, which follows from the definition of
\eqref{eq:h}. Thus we may replace the above formula for $N_{b_1}(B;36)$ by 
\[\sum_{\ee\in\EE(B)}
\frac{Y_2Y_6Y_7}{\e_5}\thtwob(\ee)\gtwob(Y_0;
  \ee;B)
+O\big(B(\log B)^4(\log\log B)\big).\]

For given $\ee\in \EE(B)$, define
\begin{align*}
\thtwo(\ee):=&
\frac{\phis(\e_1\e_2\e_3)\phis(\e_1\e_2\e_3\e_4\e_5)\phis(\e_2\e_3\e_4\e_5)}
{\zeta(2) \prod_{p\mid \e_1\e_2\e_3\e_4\e_5}(1-1/p^2)}\\
&\quad \times
\sum_{\substack{k_{12}\mid\e_2\\\cp{k_{12}}{\e_1\e_3\e_4}}}
\frac{\mu(k_{12})}{k_{12}\phis(\gcd(\e_2,k_{12}\e_5))}.
\end{align*}
It is easily seen that $\thtwo(\ee)=\thtwoa(\ee)$, in the notation
of Lemmas \ref{lem:a_e6} and \ref{lem:a_e7}. Furthermore, on noting that 
  \[\phis(\e_1\e_2\e_3)\phis(\e_1\e_2\e_3\e_4\e_5) =
  \phis(\e_1\e_2\e_3\e_4)\phis(\e_1\e_2\e_3\e_5),
\] 
since $\cp{\e_4}{\e_5}$, we see that $\thtwo(\ee)=\thtwob(\ee)$ also. 
Thus we may draw together our argument so far to conclude that 
\[N_{U,H}(B)= 
\sum_{\ee\in\EE(B)}
\frac{Y_2Y_6Y_7}{\e_5}\thtwo(\ee)\gtwo(Y_0;\ee;B)
+O\big(B(\log B)^{5-2/7}\big),\]
where $\gtwo(t_0;\ee;B)$ is given by \eqref{eq:anna-fly}.
It turns out that there is a negligible contribution to 
$N_{U,H}(B)$ from summing 
$Y_2Y_6Y_7\thtwo(\ee)\gtwo(t_0;\ee;B)/\e_5$ over small values of 
$\ee\in\EE(B)$. The $\ee$ that give the dominant contribution belong
to the set 
\[\EE^*(B):=\big\{\ee\in \ZZp^5: \text{\eqref{eq:cpe} holds, $\fbase 2 4 3 2
  3 \leq B$ and $\fbase 3 6 4 2 5 >B$}\big\}.\]
We also wish to remove the dependence on $\ee$ and $B$ from the
real-valued function $\gtwo(Y_0;\ee;B)$. All of this will be achieved
in the following result.

\begin{lemma}\label{lem:anna-swim}
We have 
\begin{align*}
N_{U,H}(B)= \omega_\infty B
\sum_{\ee\in\EE^*(B)}
\frac{\thtwo(\ee)}{\fbase 1 1 1 1 1}
+O\big(B(\log B)^{5-2/7}\big),
\end{align*}
where $\omega_\infty$ is given by \eqref{eq:inf}.
\end{lemma}

\begin{proof}
We begin by showing that 
\[M_1(B):=\sum_{\substack{\ee\in\EE(B)\\\fbase 3 6 4 2 5 \le B}}
\frac{Y_2Y_6Y_7}{\e_5}\thtwo(\ee)\gtwo(Y_0;\ee;B)
\ll B(\log B)^{4}.\]
Now it follows from \eqref{eq:g0} and \eqref{eq:anna-fly} that
\begin{align*}
\gtwo(t_0;\ee;B)
&=
\int_{h(t_0,t_2,t_6,t_7) \le 1, |Y_7t_7|>1, t_6>0} \dd t_2 \dd t_6 \dd t_7\\
&\leq 
\int_{|t_7|>1/Y_7}
\int_0^\infty g_0(t_0,t_6,t_7) \dd t_6\dd t_7.
\end{align*}
Hence Lemma~\ref{lem:bound_g}\eqref{it:bound_gonea} yields
\[\gtwo(Y_0;\ee;B) \ll \int_{|t_7| > 1/Y_7} \frac{1}{Y_0|t_7|^{7/6}} \dd t_7 \ll
\frac{Y_7^{1/6}}{Y_0}.\] 
Applying this we deduce that
\[
  \begin{split}
M_1(B) \ll \sum_{\fbase 3 6 4 2 5 \le B}
    \frac{\phid(\e_2)Y_2Y_6Y_7^{7/6}}{\e_5Y_0}
    &\ll \sum_{\fbase 3 6 4 2 5 \le B} \frac{\phid(\e_2)B^{5/6}}
    {\fbase{1/2}{0}{1/3}{2/3}{1/6}}\\
    &\ll \sum_{\e_1,\e_2,\e_3,\e_4}\frac{\phid(\e_2)B}{\fbase 1 1 1 1 0}\\
    &\ll B(\log B)^4,
  \end{split}\]
by \eqref{eq:phid}, which therefore shows that 
\begin{align*}
N_{U,H}(B)= 
\sum_{\ee\in\EE^*(B)}
\frac{Y_2Y_6Y_7}{\e_5}\thtwo(\ee)\gtwo(Y_0;\ee;B)
+O\big(B(\log B)^{5-2/7}\big).
\end{align*}
It remains to deal with the real-valued function $\gtwo(Y_0;\ee,B)$.

We will show that 
\[M_2(B):=\sum_{\ee\in\EE^*(B)}
\frac{Y_2Y_6Y_7}{\e_5}\thtwo(\ee)\int_{\substack{h(Y_0,t_2,t_6,t_7) \le 1\\
    |Y_7t_7|\leq 1, t_6>0}} \dd t_2 \dd t_6 \dd t_7 \ll B(\log B)^{4}.\] Once
achieved, this will suffice to complete the proof of the lemma, since an
application of Lemma \ref{lem:omega_infty} reveals that
\[\int_{h(Y_0,t_2,t_6,t_7) \le 1, t_6>0} \dd t_2 \dd t_6 \dd t_7=
\frac{\omega_\infty}{Y_0^2},\] and clearly
\[\frac{Y_2Y_6Y_7}{Y_0^2\e_5}=\frac{B}{\fbase 1 1 1 1 1}.\]
To establish the bound for $M_2(B)$ we appeal to
Lemma~\ref{lem:bound_g}\eqref{it:bound_gonea}, which in a similar manner to
our treatment of $M_1(B)$, implies that
\begin{align*}
  M_2(B) \ll \sum_{\ee\in\EE^*(B)}
  \frac{Y_2Y_6Y_7\phid(\e_2)}{\e_5}\int_{|t_7|\leq 1/Y_7} \frac{\dd
    t_7}{Y_0^8} &\ll \sum_{\ee\in\EE^*(B)}
  \frac{Y_2Y_6\phid(\e_2)}{Y_0^8 \e_5}\\
  & = \sum_{\ee\in\EE^*(B)}
  \frac{\phid(\e_2)B^2}{\fbase 4 7 5 3 6}\\
  & \ll \sum_{\e_1,\e_2,\e_3,\e_4}
  \frac{\phid(\e_2)B}{\fbase 1 1 1 1 0}\\
  & \ll B(\log B)^4.
\end{align*} 
This completes the proof of the lemma.
\end{proof}

Let us redefine the function $\thtwo(\ee)$ so that it is equal to zero
if $\ee$ fails to satisfy the coprimality relations in
\eqref{eq:cpe}. For $\kk = (k_1,\dots,k_5) \in \ZZp^5$, let
\[
\Delta_\kk(n):=\sum_{\substack{\ee\in \ZZp^5\\
    \fbase{k_1}{k_2}{k_3}{k_4}{k_5} = n}} 
\frac{\thtwo(\ee)}{\fbase 1 1 1 1 1}.
\] 
Then Lemma \ref{lem:anna-swim} implies that 
\begin{equation}\label{eq:anna-walk}
N_{U,H}(B)= \omega_\infty B
\sum_{n \le B}
\big(\Delta_{(2,4,3,2,3)}(n)-\Delta_{(3,6,4,2,5)}(n)\big)
+O\big(B(\log B)^{5-2/7}\big).
\end{equation}
We will want to establish an asymptotic formula for 
\[
M_\kk(t) := \sum_{n \le t} \Delta_\kk(n),
\] 
as $t\rightarrow \infty$. We shall do so by studying the corresponding 
Dirichlet series 
\[F_\kk(s) = \sum_{n=1}^\infty \frac{\Delta_\kk(n)}{n^{s}} =
\sum_{\ee\in \ZZp^5} \frac{\thtwo(\ee)}
{\fbase{k_1s+1}{k_2s+1}{k_3s+1}{k_4s+1}{k_5s+1}},
\] 
which is absolutely convergent for $\RE(s)>0$.

By multiplicativity we clearly have an Euler product $F_\kk(s)= \prod_p
F_{\kk,p}(s),$ and a cumbersome computation reveals that the local factors
$F_{\kk,p}(s)$ are equal to
\[
\begin{split}
  & (1-1/p)\cdot \left((1+1/p) +\frac{1-1/p}{p^{k_1s+1}-1} \right.
  \\&\left.+\frac{1-1/p}{p^{k_2s+1}-1}\left((1-2/p)+\frac{1-1/p}
      {p^{k_1s+1}-1}+\frac{1-1/p}{p^{k_3s+1}-1}
      +\frac{1-1/p}{p^{k_5s+1}-1}\right)\right.\\
  &\left.+\frac{(1-1/p)^2}{p^{k_3s+1}-1}\left(1+\frac{1}{p^{k_4s+1}-1}\right)+
    \frac{1-1/p}{p^{k_4s+1}-1}+\frac{1-1/p}{p^{k_5s+1}-1}\right).
\end{split}
\]
Let $\ep>0$ and assume that 
$\kk \in \{(2,4,3,2,3),(3,6,4,2,5)\}$. 
Then it follows that for all $s \in \CC$ belonging to the half-plane
$\RE(s) \ge -1/12+\ep$, we have 
\[
F_{\kk,p}(s)\prod_{j=1}^5\Big(1-\frac{1}{p^{k_js+1}}\Big) = 1+O_\ep(p^{-1-\ep}).
\] 
Thus, on defining 
\[
E_\kk(s) := \prod_{j=1}^5\zeta(k_js+1), \quad G_\kk(s) :=
\frac{F_\kk(s)}{E_\kk(s)},
\] 
we may conclude that $F_\kk(s)$ has a meromorphic 
continuation to  the half-plane $\RE(s) \ge -1/12+\ep$, 
with a pole of order $5$ at $s=0$. 
It will be useful to note that
\begin{equation}
  \label{eq:anna-run}
G_\kk(0)=\prod_p \left(1-\frac 1 p\right)^6 \left(1+\frac 6 p +\frac
  1{p^2}\right).
\end{equation}
To estimate $M_\kk(t)$ we now have everything in place to apply the
following standard Tauberian theorem,
which is recorded in work of Chambert-Loir and Tschinkel
\cite[Appendice~A]{perron}.

\begin{lemma}\label{lem:tauber}
Let $\{c_n\}_{n\in \ZZp}$ be a sequence of positive real
numbers, and let 
$f(s) =\sum_{n=1}^\infty c_n n^{-s}$. 
Assume that:
\begin{enumerate}
\item 
the series defining f(s) converges for $\RE(s) > 0$;
\item 
it admits a meromorphic continuation to $\RE(s) > -\delta$ for some
$\delta>0$, with a unique pole at $s = 0$ of order $b\in\ZZp$;
\item
there exists $\kappa> 0$ such that
\[\Big| \frac{f(s)s^b}{(s+2\delta)^b}\Big|\ll (1 + \Im m(s))^\kappa,\]
for $\RE(s) > -\delta$.
\end{enumerate}
Then there exists a monic polynomial $P$ of degree $b$, and a constant
$\delta'> 0$ such that
\[\sum_{n\leq t}c_n = \frac{\Theta}{b!} P(\log t) + O(t^{-\delta'}),\]
as $t\rightarrow \infty$, where $\Theta=\lim_{s\rightarrow 0} s^bf(s)$.
\end{lemma}

In fact \cite[Appendice~A]{perron} deals only with Dirichlet series possessing
a unique pole at $s=a>0$, but the extension to a pole at $s=0$ is
straightforward.  We apply Lemma \ref{lem:tauber} to estimate $M_\kk(t)$,
for \[\kk \in \{(2,4,3,2,3),(3,6,4,2,5)\}.\] We have already seen that the
corresponding Dirichlet series $F_\kk(s)$ satisfies parts (1) and (2) of the
lemma, with $b=5$.  The third part follows from the boundedness of
$G_{\kk}(s)$ on the half-plane $\RE(s) \ge -1/12+\ep$, and standard upper
bounds for the size of the Riemann zeta function in the critical strip.  In
view of the fact that
\[\lim_{s\rightarrow 0} s^bF(s) = \frac{G_\kk(0)}{\prod_{j=1}^5 k_j},\]
we therefore conclude that 
\begin{equation}
  \label{eq:perron}
  M_\kk(t) = \frac{G_\kk(0)P(\log t)}{5! \cdot \prod_{j=1}^5 k_j}
  +O(t^{-\delta}),
\end{equation}
for some $\delta>0$ and some monic polynomial $P$ of degree $5$.

We are now ready to complete the proof of our theorem.  Recall the definition
\eqref{eq:om} of $\omega_H(\tS)$.  It therefore follows on combining
\eqref{eq:alpha}, \eqref{eq:anna-walk}, \eqref{eq:anna-run} and
\eqref{eq:perron} that \[N_{U,H}(B)= \alpha(\tS)\omega_H(\tS)B (\log B)^5
+O\big(B(\log B)^{5-2/7}\big),\] as required.

\end{document}